\documentclass[fleqn,11pt,letterpaper]{article}
\usepackage{amsfonts}
\usepackage{amssymb}
\usepackage{amsmath}
\usepackage{amsthm}
\usepackage{enumitem}
\usepackage{multirow}
\usepackage{float}
\usepackage{graphicx}
\usepackage{color}
\usepackage{stmaryrd}
\usepackage[margin=2.57cm]{geometry}
\definecolor{darkred}{rgb}{0.5,0.2,0.2}
\usepackage[pdftitle={Hybrid scheme for Brownian semistationary processes},pdfauthor={Mikkel Bennedsen, Asger Lunde and Mikko Pakkanen},colorlinks=TRUE,allcolors=darkred]{hyperref}
\usepackage{booktabs}

\usepackage{float}
\floatstyle{plaintop}
\restylefloat{table}

\restylefloat{table}

\numberwithin{equation}{section}

\theoremstyle{plain}
\newtheorem{theorem}{Theorem}[section]
\newtheorem{corollary}[theorem]{Corollary}
\newtheorem{lemma}[theorem]{Lemma}
\newtheorem{proposition}[theorem]{Proposition}
\theoremstyle{definition}
\newtheorem{example}[theorem]{Example}
\theoremstyle{remark}
\newtheorem{remark}[theorem]{Remark}

\newcommand{\N}{\mathbb{N}}

\newcommand{\R}{\mathbb{R}}

\newcommand{\E}{\mathbb{E}}
\newcommand{\BSS}{\mathcal{BSS}}

\newcommand\mikkel[1]{#1}

\def\bi{\begin{itemize}}
\def\ei{\end{itemize}}

\allowdisplaybreaks
\graphicspath{{Figures/}}

\newif\ifi

\begin{document}

\title{Hybrid scheme for Brownian semistationary processes}

\author{Mikkel Bennedsen\thanks{
Department of Economics and Business Economics and CREATES, 
Aarhus University, 
Fuglesangs All\'e 4,
8210 Aarhus V, Denmark.
E-mail:\
\href{mailto:mbennedsen@econ.au.dk}{\nolinkurl{mbennedsen@econ.au.dk}} 
} \and 
Asger Lunde\thanks{
Department of Economics and Business Economics and CREATES, 
Aarhus University, 
Fuglesangs All\'e 4,
8210 Aarhus V, Denmark.
E-mail:\ 
\href{mailto:alunde@econ.au.dk}{\nolinkurl{alunde@econ.au.dk}} 
} \and Mikko S. Pakkanen\thanks{
Department of Mathematics, 
Imperial College London, 
South Kensington Campus,
London SW7 2AZ, UK and CREATES, Aarhus University, Denmark.
E-mail:\ 
\href{mailto:m.pakkanen@imperial.ac.uk}{\nolinkurl{m.pakkanen@imperial.ac.uk}}
}}

\maketitle

\begin{abstract}
We introduce a simulation scheme for Brownian semistationary processes, which is based on discretizing the stochastic integral representation of the process in the time domain. We assume that the kernel function of the process is regularly varying at zero. The novel feature of the scheme is to approximate the kernel function by a power function near zero and by a step function elsewhere. The resulting approximation of the process is a combination of Wiener integrals of the power function and a Riemann sum, which is why we call this method a hybrid scheme. Our main theoretical result describes the asymptotics of the mean square error of the hybrid scheme and we observe that the scheme leads to a substantial improvement of accuracy compared to the ordinary forward Riemann-sum scheme, while having the same computational complexity.   
We exemplify the use of the hybrid scheme by two numerical experiments, where we 
examine the finite-sample properties of an estimator of the roughness parameter of a Brownian semistationary process and
study Monte Carlo option pricing in the rough Bergomi model of Bayer et al.\ \cite{BFG2015}, respectively.
\end{abstract}


\vspace*{1em}


\vspace*{1em}


\section{Introduction}

We study simulation methods for \emph{Brownian semistationary} ($\BSS$) processes, first introduced by Barndorff-Nielsen and Schmiegel \cite{ambitTC07,ole-jurgen09}, which form a flexible class of stochastic processes that are able to capture some common features of empirical time series, such as stochastic volatility (intermittency), roughness, stationarity and strong dependence. \mikkel{By now these processes have been applied} in various contexts, most notably in the study of turbulence in physics \cite{ole_mikko_jurgen13,jose_emil_mikko_mark13} and in finance as models of energy prices \cite{ole_fred_almut13,bennedsen15}. \mikkel{A} $\BSS$ process $X$ is defined via the integral representation
\begin{align}\label{eq:BSS1}
X(t) = \int_{-\infty}^t g(t-s)\sigma(s) dW(s),
\end{align}
where $W$ is a two-sided Brownian motion providing the fundamental noise innovations, the amplitude of which is modulated by a stochastic volatility (intermittency) process $\sigma$ that may depend on $W$. This driving noise is then convolved with a deterministic kernel function $g$ that specifies the \mikkel{dependence structure of $X$}. The process $X$ can also be viewed as a \emph{moving average} of volatility-modulated Brownian noise and setting $\sigma(s) = 1$, we see that stationary \emph{Brownian moving averages} are nested in this class of processes.

In the applications mentioned above, the \mikkel{case where $X$ is not a} semimartingale is particularly relevant. This situation arises when the kernel function $g$ behaves like a power-law near zero; more specifically, \mikkel{when} for some $\alpha \in ( -\frac{1}{2},\frac{1}{2}) \setminus \{0 \}$,
\begin{equation}\label{eq:powerlawatzero}
g(x) \propto x^\alpha \quad \textrm{for small $x>0$.}
\end{equation}
Here we write ``$\propto$'' to indicate proportionality in an \emph{informal} sense, anticipating a rigorous formulation of this relationship given in Section \ref{ssec:kernel} using the theory of \emph{regular variation} \cite{BGT}, which plays a significant role in our subsequent arguments. 
The case $\alpha = -\frac{1}{6}$ in \eqref{eq:powerlawatzero} is important in statistical modeling of turbulence \cite{jose_emil_mikko_mark13} as it gives rise to processes that are compatible with Kolmogorov's scaling law for ideal turbulence. Moreover, processes of similar type with $\alpha \approx - 0.4$ have been recently used in the context of option pricing as models of \emph{rough volatility} \cite{ALV2007,BFG2015,Fu2016,GJR2014}, see Sections \ref{sec:tbss} and \ref{sec:roughvol} below.
The case $\alpha=0$ would (roughly speaking) lead to a process that is a semimartingale, which is thus excluded.

Under \eqref{eq:powerlawatzero}, the trajectories of $X$ behave locally like the trajectories of a \emph{fractional Brownian motion} with Hurst index $H = \alpha + \frac{1}{2} \in (0,1) \setminus \{ \frac{1}{2} \}$. While the \emph{local} behavior and roughness, measured in terms of H\"older regularity, of $X$ are determined by the parameter $\alpha$, the \emph{global} behavior of $X$ (e.g., whether the process has long or short memory) depends on the behavior of $g(x)$ as $x \rightarrow \infty$, which can be specified independently of $\alpha$.
 This should be contrasted with fractional Brownian motion and related \emph{self-similar} models, which necessarily must conform to a restrictive affine relationship between their H\"older regularity  (local behavior and roughness) and  Hurst index (global behavior), as elucidated by Gneiting and Schlather \cite{gneiting04}. Indeed, in the realm of $\BSS$ processes, local and global behavior are conveniently decoupled, which underlines the flexibility of these processes as a modeling framework.
 
In connection with practical applications, it is important to be able to simulate the process $X$. If the volatility process $\sigma$ \mikkel{is deterministic} and constant in time, then $X$ \mikkel{will be} strictly stationary and Gaussian. \mikkel{This makes} $X$ amenable to exact simulation using the Cholesky factorization or circulant embeddings, see, e.g., \cite[Chapter XI]{AsmussenGlynn07}. However, it seems difficult, if not impossible, to develop an exact method \mikkel{that is applicable} with a stochastic $\sigma$, as the process $X$ is then neither Markovian nor Gaussian. 
Thus, in the general case \mikkel{one must} resort to approximative methods. To this end, Benth et al.\ \cite{heidar13fourier} have recently proposed a Fourier-based method of simulating $\mathcal{BSS}$ processes, and more general \emph{L\'evy semistationary} ($\mathcal{LSS}$) processes, which relies on approximating the kernel function $g$ in the frequency domain.

In this paper, we introduce a new discretization scheme for $\mathcal{BSS}$ processes based on approximating the kernel function $g$ in the time domain. Our starting point is the Riemann-sum discretization of \eqref{eq:BSS1}. The Riemann-sum scheme builds on an approximation of $g$ using step functions, which has the pitfall of failing to capture appropriately the steepness of $g$ near zero. In particular, this becomes a serious defect under \eqref{eq:powerlawatzero} when $\alpha \in ( -\frac{1}{2},0)$. In our new scheme, we mitigate this problem by approximating $g$ using an appropriate power function near zero and a step function elsewhere. The resulting discretization scheme can be realized as a linear combination of Wiener integrals with respect to the driving Brownian motion $W$ and a Riemann sum, which is why we call it a \emph{hybrid scheme}. The hybrid scheme is only slightly more demanding to implement than the Riemann-sum scheme and the schemes have the same computational complexity as the number of discretization cells tends to infinity.

Our main theoretical result describes the exact asymptotic behavior of the mean square error (MSE) of the hybrid scheme and, as a special case, that of the Riemann-sum scheme. We observe that switching from the Riemann-sum scheme to the hybrid scheme reduces the asymptotic root mean square error (RMSE) substantially. Using merely the simplest variant the of hybrid scheme, where a power function is used in a single discretization cell, the reduction is at least $50\%$ for all $\alpha \in ( 0,\frac{1}{2})$ and at least $80\%$ for all $\alpha \in ( -\frac{1}{2},0)$. The reduction in RMSE is close to $100\%$ as $\alpha$ approches $-\frac{1}{2}$, which indicates that the hybrid scheme indeed resolves the problem of poor precision that affects the Riemann-sum scheme.

To assess the accuracy of the hybrid scheme in practice, we perform two numerical experiments. Firstly, we examine the finite-sample performance of an estimator of the roughness index $\alpha$, introduced by Barndorff-Nielsen et al.\ \cite{ole_jose_mark13} and Corcuera et al.\ \cite{jose_emil_mikko_mark13}. This experiment enables us to assess how faithfully the hybrid scheme approximates the fine properties of the $\BSS$ process $X$. Secondly, we study Monte Carlo option pricing in the rough Bergomi stochastic volatility model of Bayer et al.\ \cite{BFG2015}. We use the hybrid scheme to simulate the volatility process in this model and we find that the resulting implied volatility smiles are indistinguishable from those simulated using a method that involves exact simulation of the volatility process. Thus we are able propose a solution to the problem of finding an efficient simulation scheme for the rough Bergomi model, left open in the paper \cite{BFG2015}.

The rest of this paper is organized as follows. In Section \ref{sec:model} we recall the rigorous definition of a $\BSS$ process and introduce our assumptions. We also introduce the hybrid scheme, state our main theoretical result concerning the asymptotics of the mean square error and discuss an extension of the scheme to a class of \emph{truncated} $\mathcal{BSS}$ processes. Section \ref{sec:app} briefly discusses the implementation of the discretization scheme and presents the numerical experiments mentioned above. Finally, Section \ref{sec:proofs} contains the proofs of the theoretical and technical results given in the paper.

\section{The model and theoretical results}\label{sec:model}

\subsection{Brownian semistationary process}

Let $(\Omega,\mathcal{F},\{\mathcal{F}_t\}_{t \in \R}, \mathbb{P})$ be a filtered probability space, satisfying the usual conditions, \mikkel{supporting a} (two-sided) standard Brownian motion $W=\{W(t)\}_{t \in \R}$. We consider a Brownian semistationary process
\begin{align}\label{eq:X}
X(t) = \int_{-\infty}^t g(t-s)\sigma(s) dW(s), \quad t \in \R,
\end{align}
where $\sigma = \{\sigma(t)\}_{t\in \R}$ is an $\{\mathcal{F}_t\}_{t \in \R}$-predictable process with locally bounded trajectories, which captures the stochastic volatility (intermittency) of $X$, and $g: (0,\infty) \to [0,\infty)$ is a Borel measurable kernel function.

 To ensure that the integral \eqref{eq:X} is well-defined, we assume that the kernel function $g$ is square integrable, that is, $\int_0^\infty g(x)^2 dx < \infty$. \mikkel{In fact, we will shortly introduce some more specific assumptions on $g$ that will imply its square integrability}. Throughout the paper, we also assume that the process $\sigma$ has finite second moments, $\E[\sigma(t)^2] < \infty$ for all $t \in \R$, and that the process is covariance stationary, namely,
 \begin{equation*}
\mathbb{E}[\sigma(s)] = \mathbb{E}[\sigma(t)],\quad  \mathrm{Cov}(\sigma(s),\sigma(t)) = \mathrm{Cov}(\sigma(0),\sigma(|s-t|)), \quad s,t \in \R.
 \end{equation*}
These assumptions imply that also $X$ is covariance stationary, \mikkel{that is},
\begin{equation*}
\mathbb{E}[X(t)] = 0, \quad \mathrm{Cov}(X(s),X(t)) = \E[\sigma(0)^2]\int_0^\infty g(x)g(x+|s-t|)dx, \quad s,t \in \R.
\end{equation*}
However, the process $X$ need not be \emph{strictly} stationary as the dependence between the volatility process $\sigma$ and the driving Brownian motion $W$ may be time-varying.

\subsection{Kernel function}\label{ssec:kernel}
As mentioned above, we consider a kernel \mikkel{function that} satisfies $g(x) \propto x^\alpha$ for some $\alpha \in (-\frac{1}{2},\frac{1}{2})\setminus \{0 \}$ when $x>0$ is near zero. To make this idea rigorous and to allow for additional flexibility, we formulate our assumptions on $g$ using the theory of regular variation \cite{BGT} and, more specifically, slowly varying functions.

 To this end, recall that a measurable function $L : (0,1] \rightarrow [0,\infty)$ is \emph{slowly varying} at $0$ if for any $t>0$,
\begin{align*}
\lim_{x \rightarrow 0} \frac{L(tx)}{L(x)} = 1.
\end{align*}
Moreover, a function $f(x) = x^\beta L(x)$, $x \in (0,1]$, where $\beta \in \R$ and $L$ is slowly varying at $0$, is said to be \emph{regularly varying} at $0$, with $\beta$ being the \emph{index of regular variation}. 

\begin{remark}
Conventionally, slow and regular variation are defined at $\infty$ \cite[pp.\ 6,\ 17--18]{BGT}. However, $L$ is slowly varying (resp.\ regularly varying) at $0$ if and only if $x \mapsto L(1/x)$ is slowly varying (resp.\ regularly varying) at $\infty$.
\end{remark}

A key feature of slowly varying functions, which will be very important in the sequel, is that they can be sandwiched between polynomial functions as follows. If $\delta>0$ and $L$ is slowly varying at $0$ and bounded away from $0$ and $\infty$ on any interval $(u,1]$, $u \in (0,1)$, then there exist constants $\overline{C}_{\delta} \geq \underline{C}_{\delta} >0$ such that 
\begin{align}\label{eq:potterzero}
\underline{C}_{\delta} x^{\delta} \leq L(x) \leq \overline{C}_{\delta} x^{-\delta}, \quad x \in (0,1].
\end{align}
The inequalities above are an immediate consequence of the so-called \emph{Potter bounds} for slowly varying functions, see \cite[Theorem 1.5.6(ii)]{BGT} and \eqref{eq:potterbounds} below. Making $\delta$ very small therein, we see that slowly varying functions are asymptotically negligible in comparison with polynomially growing/decaying functions. Thus, by multiplying power functions and slowly varying functions, regular variation provides a flexible framework to construct functions that behave asymptotically like power functions.

Our assumptions concerning the kernel function $g$ are as follows: 
\begin{enumerate}[label=(A\arabic*),ref=A\arabic*,leftmargin=3em]
\item \label{ass:zero}
For some $\alpha \in (-\frac{1}{2},\frac{1}{2}) \setminus \{0\}$,
\begin{align*}
g(x) = x^{\alpha}L_g(x),	\quad x \in (0,1],
\end{align*}
where $L_g : (0,1] \to [0,\infty)$ is continuously differentiable, slowly varying at $0$ and bounded away from $0$. Moreover, there exists a constant $C>0$ such that the derivative $L'_g$ of $L_g$ satisfies
\begin{align*}
	|L_g'(x)| \leq C(1+x^{-1}), \quad x\in (0,1].
\end{align*}
\item \label{ass:deriv}
	The function $g$ is continuously differentiable on $(0,\infty)$, so that its derivative $g'$ is ultimately monotonic and satisfies $\int_1^{\infty}g'(x)^2dx<\infty$.
\item\label{ass:infb} For some $\beta \in (-\infty,-\frac{1}{2})$,
\begin{equation*}
g(x) = \mathcal{O}(x^{\beta}), \quad x \rightarrow \infty.
\end{equation*}
\end{enumerate}
(Here, and in the sequel, we use $f(x) = \mathcal{O}(h(x))$, $x \rightarrow a$, to indicate that $\limsup_{x \rightarrow a} \big|\frac{f(x)}{h(x)}\big| < \infty$. Additionally, analogous notation is later used for sequences and computational complexity.)
In view of the bound \eqref{eq:potterzero}, these assumptions ensure that $g$ is square integrable. It is worth pointing out that \eqref{ass:zero} accommodates functions $L_g$ with $\lim_{x \rightarrow 0} L_g(x) = \infty$, e.g., $L_g(x) = 1 - \log x$.

The assumption \eqref{ass:zero} influences the short-term behavior and roughness of the process $X$. A simple way to assess the roughness of $X$ is to study the behavior of its \emph{variogram} (also called the \emph{second-order structure function} in turbulence literature)
\begin{equation*}
V_X(h) := \mathbb{E}[|X(h) - X(0)|^2], \quad h \geq 0,
\end{equation*}
as $h \rightarrow 0$. Note that, by covariance stationarity,
\begin{equation*}
V_X(|s-t|) = \mathbb{E}[|X(s) - X(t)|^2], \quad s,t \in \R.
\end{equation*}
Under our assumptions, we have the following characterization of the behavior of $V_X$ near zero, which generalizes a result of Barndorff-Nielsen \cite[p.\ 9]{olegamma12}
and implies that $X$ has a locally H\"older continuous modification. Therein, and in what follows, we write $a(x) \sim b(x)$, $x \rightarrow y$, to indicate that $\lim_{x \rightarrow y} \frac{a(x)}{b(x)} = 1$.
The proof of this result is carried out in Section \ref{ssec:proofs1}.

\begin{proposition}[Local behavior and continuity]\label{prop:localbehavior}
Suppose that \eqref{ass:zero}, \eqref{ass:deriv} and \eqref{ass:infb} hold.
\begin{enumerate}[label=(\roman*),ref=\roman*,leftmargin=*]
\item\label{item:vario} The variogram of $X$ satisfies
\begin{equation*}
V_X(h) \sim \E[\sigma(0)^2] \bigg(\frac{1}{2\alpha + 1} + \int_0^\infty \big( (y+1)^\alpha - y^\alpha\big)^2 dy\bigg) h^{2\alpha+1} L_g(h)^2, \quad h \rightarrow 0,
\end{equation*}
which implies that $V_X$ is regularly varying at zero with index $2\alpha+1$.
\item\label{item:holder} The process $X$ has a modification with locally $\phi$-H\"older continuous trajectories for any $\phi \in (0,\alpha + \frac{1}{2})$.
\end{enumerate}
\end{proposition}

Motivated by Proposition \ref{prop:localbehavior}, we call $\alpha$ the \emph{roughness index} of the process $X$.
Ignoring the slowly varying factor \mikkel{$L_g(h)^2$  in} \eqref{prop:localbehavior}, we see that the variogram $V(h)$ behaves like $h^{2\alpha+1}$ for small values of $h$, which is reminiscent of the scaling property of the increments of a fractional Brownian motion (fBm) with Hurst index $H = \alpha + \frac{1}{2}$. Thus, the process $X$ behaves locally like such an fBm, at least when it comes to second order structure and roughness. (Moreover, the factor $\frac{1}{2\alpha + 1} + \int_0^\infty ( (y+1)^\alpha - y^\alpha)^2 dy$ coincides with the normalization coefficient that appears in the Mandelbrot--Van Ness representation \cite[Theorem 1.3.1]{Mishura2008} of an fBm with $H = \alpha + \frac{1}{2}$.)

Let us now look at two examples of a kernel function $g$ that satisfies our assumptions.

\begin{example}[The gamma kernel]\label{ex:gamma}
The so-called \emph{gamma kernel} 
\begin{align*}
g(x) = x^{\alpha} e^{-\lambda x}, \quad x \in (0,\infty),
\end{align*}
with parameters $\alpha \in (-\frac{1}{2},\frac{1}{2}) \setminus \{ 0 \}$ and $\lambda > 0$, has been used extensively in the literature on $\mathcal{BSS}$ processes. It is particularly important in connection with statistical modeling of turbulence, see Corcuera et al.\ \cite{jose_emil_mikko_mark13}, but it also provides a way to construct generalizations of Ornstein--Uhlenbeck (OU) processes with roughness that differs from the usual semimartingale case $\alpha = 0$, while mimicking the long-term behavior of an OU process. Moreover, $\mathcal{BSS}$ and $\mathcal{LSS}$ processes defined using the gamma kernel have interesting probabilistic properties, see \cite{PS2015}. An in-depth study of the gamma kernel can be found in \cite{olegamma12}. Setting $L_g (x) := e^{-\lambda x}$, which is slowly varying at $0$ since $\lim_{x \rightarrow 0} L_g(x) = 1$, it is evident that \eqref{ass:zero} holds. Since $g(x)$ decays exponentially fast to $0$ as $x \rightarrow \infty$, it is clear that also \eqref{ass:infb} holds.
To verify \eqref{ass:deriv}, note that $g$ satisfies 
\begin{align*}
g'(x) = \bigg( \frac{\alpha}{x}-\lambda\bigg) g(x), \quad g''(x) = \Bigg(\bigg(\frac{\alpha}{x}-\lambda\bigg)^2- \frac{\alpha}{x^2}\Bigg) g(x), \quad x \in (0,\infty),
\end{align*}
where $\lim_{x \rightarrow \infty }((\frac{\alpha}{x}-\lambda)^2- \frac{\alpha}{x^2})= \lambda^2>0$, so $g'$ is ultimately increasing with 
\begin{align*}
g'(x)^2 \leq (|\alpha| + \lambda)^2 g(x)^2, \quad x \in [1,\infty). 
\end{align*}
Thus, $\int_1^\infty g'(x)^2 dx < \infty$ since $g$ is square integrable.
\end{example}

\begin{example}[Power-law kernel]
Consider the kernel function
\begin{align*}
g(x) = x^\alpha (1+x)^{\beta-\alpha}, \quad x \in (0,\infty),
\end{align*}
with parameters $\alpha \in (-\frac{1}{2},\frac{1}{2})\setminus \{0\}$ and $\beta \in (-\infty,-\frac{1}{2})$. The behavior of this kernel function near zero is similar to that of the gamma kernel, but $g(x)$ decays to zero polynomially as $x \rightarrow \infty$, so it can be used to model long memory. In fact, it can be shown that if $\beta \in (-1,-\frac{1}{2})$, then the autocorrelation function of $X$ is not integrable. Clearly, \eqref{ass:zero} holds with $L_g(x) := (1+x)^{\beta-\alpha}$, which is slowly varying at $0$ since $\lim_{x \rightarrow 0} L_g(x) = 1$. Moreover, note that we can write
\begin{equation*}
g(x) = x^\beta K_g(x), \quad x \in (0,\infty),
\end{equation*}
where $K_g(x):=(1+x^{-1})^{\beta-\alpha}$ satisfies $\lim_{x \rightarrow \infty}K_g(x)=1$. Thus, also \eqref{ass:infb} holds. We can check \eqref{ass:deriv} by computing
\begin{equation*}
g'(x) = \bigg( \frac{\alpha+\beta x}{x(1+x)} \bigg)g(x), \quad g''(x) = \Bigg(\bigg(\frac{\alpha+\beta x}{x(1+x)}\bigg)^2 + \frac{-\alpha-2\alpha x- \beta x^2}{x^2(1+x)^2}\Bigg) g(x), \quad x \in (0,\infty),
\end{equation*}
where $-\alpha-2\alpha x- \beta x^2 \rightarrow \infty$ when $x \rightarrow \infty$ (as $\beta < -\frac{1}{2}$), so $g'$ is ultimately increasing. Additionally, we note that
\begin{equation*}
g'(x)^2 \leq (|\alpha|+|\beta|)^2 g(x)^2, \quad x \in [1,\infty),
\end{equation*}
implying $\int_1^\infty g'(x)^2 dx < \infty$ since $g$ is square integrable.
\end{example}

\subsection{Hybrid scheme}

Let $t \in \R$ and consider discretizing $X(t)$ based on its integral representation \eqref{eq:X} on the grid $\mathcal{G}_n(t) := \{t,t-\frac{1}{n}, t-\frac{2}{n},\ldots\}$ for $n \in \N$. To derive our discretization scheme, let us first note that if the volatility process $\sigma$ does not vary too much, then it is reasonable to use the approximation
\begin{equation}\label{eq:sigmablock}
X(t) = \sum_{k = 1}^\infty  \int_{t-\frac{k}{n}}^{t-\frac{k}{n}+\frac{1}{n}} g(t-s) \sigma (s) dW(s)  \approx \sum_{k = 1}^\infty \sigma \bigg(t-\frac{k}{n}\bigg) \int_{t-\frac{k}{n}}^{t-\frac{k}{n}+\frac{1}{n}} g(t-s)  dW(s),
\end{equation}
that is, we keep $\sigma$ constant in each discretization cell. (Here, and in the sequel, ``$\approx$'' stands for an informal approximation used for purely heuristic purposes.)
If $k$ is ``small'', then due to \eqref{ass:zero} we may approximate
\begin{equation}\label{eq:smallapp}
g(t-s) \approx (t-s)^\alpha L_g\bigg( \frac{k}{n} \bigg),\quad t-s \in \bigg[\frac{k-1}{n},\frac{k}{n}\bigg]\setminus \{ 0 \},
\end{equation}
as the slowly varying function $L_g$ varies ``less'' than the power function $y \mapsto y^\alpha$ near zero, cf.\ \eqref{eq:potterzero}. If $k$ is ``large'', or at least $k \geq 2$, then choosing $b_k \in [k-1,k]$ provides an adequate approximation
\begin{equation}\label{eq:largeapp}
g(t-s) \approx g\bigg(\frac{b_k}{n}\bigg),\quad t-s \in \bigg[\frac{k-1}{n},\frac{k}{n}\bigg],
\end{equation}
by \eqref{ass:deriv}.
 Applying \eqref{eq:smallapp} to the first $\kappa$ terms, where $\kappa = 1,2,\ldots$, and \eqref{eq:largeapp} to the remaining terms in the approximating series in \eqref{eq:sigmablock} yields
\begin{equation}\label{eq:combapp}
\begin{split}
\sum_{k = 1}^\infty \sigma \bigg(t-\frac{k}{n}\bigg) \int_{t-\frac{k}{n}}^{t-\frac{k}{n}+\frac{1}{n}} g(t-s)  dW(s) &  \approx  \sum_{k = 1}^\kappa L_g \bigg( \frac{k}{n}\bigg) \sigma \bigg(t-\frac{k}{n}\bigg)  \int_{t-\frac{k}{n}}^{t-\frac{k}{n}+\frac{1}{n}} (t-s)^\alpha  dW(s) \\ &
\quad + \sum_{k = \kappa + 1}^\infty g\bigg(\frac{b_k}{n}\bigg) \sigma \bigg(t-\frac{k}{n}\bigg) \int_{t-\frac{k}{n}}^{t-\frac{k}{n}+\frac{1}{n}}  dW(s),
\end{split}
\end{equation}
For completeness, we also allow for $\kappa = 0$, in which case we require that $b_1 \in (0,1]$ and interpret the first sum on the right-hand side of \eqref{eq:combapp} as zero.
To make numerical implementation feasible, we truncate the second sum on the right-hand side of \eqref{eq:combapp} so that both sums have $N_n \geq \kappa+1$ terms in total. Thus, we arrive at a discretization scheme for $X(t)$, which we call a \emph{hybrid scheme}, given by
\begin{equation*}
X_n(t) := \check{X}_n(t) + \hat{X}_n(t),
\end{equation*}
where
\begin{align}
\check{X}_n(t) & := \sum_{k=1}^\kappa L_g\bigg( \frac{k}{n} \bigg) \sigma\bigg( t-\frac{k}{n}\bigg) \int_{t-\frac{k}{n}}^{t-\frac{k}{n}+\frac{1}{n}} (t-s)^\alpha dW(s),\label{eq:checkX} \\
\hat{X}_n(t) & := \sum_{k=\kappa+1}^{N_n} g\bigg(\frac{b_k}{n}\bigg)\sigma\bigg(t- \frac{k}{n}\bigg)\Bigg( W\bigg(t-\frac{k}{n}+\frac{1}{n}\bigg) - W\bigg(t-\frac{k}{n}\bigg)\Bigg),\label{eq:hatX}
\end{align}
and $\mathbf{b}:=\{b_k\}_{k=\kappa+1}^{\infty}$ is a sequence of real numbers, evaluation points, that must satisfy $b_k \in [k-1,k]\setminus \{0\}$ for each $k\geq \kappa+1$, but otherwise can be chosen freely.

As it stands, the discretization grid $\mathcal{G}_n(t)$ depends on the time $t$, which may seem cumbersome with regard to sampling $X_n(t)$ simultaneously for different times $t$. However, note that whenever times $t$ and $t'$ are separated by a multiple of $\frac{1}{n}$, the corresponding grids $\mathcal{G}_n(t)$ and $\mathcal{G}_n(t')$ will intersect. In fact the hybrid scheme defined by \eqref{eq:checkX} and \eqref{eq:hatX} can be implemented efficiently, as we shall see in Section \ref{ssec:impl}, below.
 Since
\begin{equation*}
g\bigg(\frac{b_k}{n}\bigg) = g\bigg(t-\bigg(t-\frac{b_k}{n}\bigg)\bigg),
\end{equation*}
the degenerate case $\kappa=0$ with $b_k = k$ for all $k \geq 1$ corresponds to the usual Riemann-sum discretization scheme of $X(t)$ with (It\=o type) forward sums from \eqref{eq:hatX}. Henceforth, we denote the associated sequence $\{ k\}_{k=\kappa+1}^\infty$ by $\mathbf{b}_{\mathrm{FWD}}$, where the subscript ``$\mathrm{FWD}$'' alludes to forward sums.
 However, including terms involving Wiener integrals of a power function given by \eqref{eq:checkX}, that is having $\kappa \geq 1$, improves the accuracy of the discretization considerably, as we shall see. Having the leeway to select $b_k$ within the interval $[k-1,k]\setminus \{ 0 \}$, so that the function $g(t-\cdot)$ is evaluated at a point that does not necessarily belong to $\mathcal{G}_n(t)$, leads additionally to a moderate improvement.

The trunction in the sum \eqref{eq:hatX} entails that the stochastic integral \eqref{eq:X} defining $X$ is truncated at $t - \frac{N_n}{n}$.
In practice, the value of the parameter $N_n$ should be large enough to mitigate the effect of truncation. To ensure that the truncation point $t - \frac{N_n}{n}$ tends to $-\infty$ as $n \rightarrow \infty$ in our asymptotic results, we introduce the following assumption:
\begin{enumerate}[label=(A\arabic*),ref=A\arabic*,leftmargin=*,resume]
\item\label{ass:truncation} For some $\gamma>0$,
\begin{align*}
N_n \sim n^{\gamma + 1},\quad n \rightarrow \infty.
\end{align*}
\end{enumerate}

\subsection{Asymptotic behavior of mean square error}

We are now \mikkel{ready} to state our main theoretical result, which gives a sharp description of the asymptotic behavior of the mean square error (MSE) of the hybrid scheme as $n \rightarrow \infty$. We defer the proof of this result to Section \ref{ssec:proofs2}.

\begin{theorem}[Asymptotics of mean square error]\label{th:mainTh}
Suppose that \eqref{ass:zero}, \eqref{ass:deriv}, \eqref{ass:infb} and \eqref{ass:truncation} hold, so that
\begin{align}\label{eq:gammaasymp}
\gamma > -\frac{2\alpha+1}{2\beta +1},
\end{align}
and that for some $\delta>0$,
\begin{align}\label{eq:sigmab1}
\mathbb{E}[|\sigma(s) - \sigma(0)|^2] = \mathcal{O}\big(s^{2\alpha +1+\delta}\big),  \quad s\downarrow 0.
\end{align}
Then for all $t\in \mathbb{R}$,
\begin{align}\label{eq:rate1}
\mathbb{E}[|X(t) - X_n(t)|^2] \sim J(\alpha,\kappa,\mathbf{b}) \mathbb{E}[\sigma(0)^2] n^{-(2\alpha +1)} L_g(1/n)^2 , \quad n\rightarrow \infty,
\end{align}
where 
\begin{align}\label{eq:Jdef}
J(\alpha,\kappa,\mathbf{b}) := \sum_{k=\kappa+1}^{\infty} \int_{k-1}^k (y^{\alpha} - b_k^{\alpha})^2dy < \infty.
\end{align}
\end{theorem}

\begin{remark}
Note that if $\alpha \in (-\frac{1}{2},0)$, then having
\begin{align*}
\mathbb{E}[|\sigma(s) - \sigma(0)|^2] = \mathcal{O}\big(s^\theta\big),  \quad s\downarrow 0,
\end{align*}
for all $\theta \in (0,1)$, ensures that \eqref{eq:sigmab1} holds. (Take, say, $\delta := \frac{1}{2}(1 - (2\alpha+1))>0$ and $\theta := 2\alpha +1 + \delta = \alpha + 1  \in (0,1)$.)
\end{remark}

When the hybrid scheme is used to simulate the $\mathcal{BSS}$ process $X$ on an equidistant grid $\{0,\frac{1}{n},\frac{2}{n},\ldots,\frac{\lfloor nT \rfloor}{n} \}$ for some $T>0$ (see Section \ref{ssec:impl} on the details of the implementation), the following consequence of Theorem \ref{th:mainTh} ensures that the covariance structure of the simulated process approximates that of the actual process $X$.

\begin{corollary}[Covariance structure]\label{cor:covariance}
Suppose that the assumptions of Theorem \ref{th:mainTh} hold. Then for any $s$,\, $t\in\R$ and $\varepsilon>0$,
\begin{equation*}
|\E[X_n(t) X_n(s)]-\E[X(t) X(s)]| = \mathcal{O}\big(n^{-(\alpha +\frac{1}{2})+\varepsilon}\big), \quad n \rightarrow \infty.
\end{equation*}
\end{corollary}

\begin{proof}
Let $s$,\, $t\in\R$. Applying the Cauchy--Schwarz inequality, we get
\begin{multline*}
|\E[X_n(t) X_n(s)]-\E[X(t) X(s)]| \\  \leq \E[X_n(t)^2]^{1/2} \E[|X(s)-X_n(s)|^2]^{1/2}  + \E[X(s)^2]^{1/2} \E[|X(t)-X_n(t)|^2]^{1/2}.
\end{multline*}
We have $\sup_{n \in \N} \E[X_n(t)^2]^{1/2} <\infty$ since $\E[X_n(t)^2] \rightarrow \E[X(t)^2]<\infty$ as $n \rightarrow \infty$, by Theorem \ref{th:mainTh}. Moreover, Theorem \ref{th:mainTh} and the bound \eqref{eq:potterzero} imply that $\E[|X(s)-X_n(s)|^2]^{1/2}=\mathcal{O}(n^{-(\alpha +\frac{1}{2})+\varepsilon})$ and $\E[|X(t)-X_n(t)|^2]^{1/2}=\mathcal{O}(n^{-(\alpha +\frac{1}{2})+\varepsilon})$ for any $\varepsilon>0$.
\end{proof}

In Theorem \ref{th:mainTh}, the asymptotics of the MSE \eqref{eq:rate1} are determined by the behavior of the kernel function $g$ near zero, as specified in \eqref{ass:zero}. The condition \eqref{eq:gammaasymp} ensures that error from approximating $g$ near zero is asymptotically larger than the error induced by the truncation of the stochastic integral \eqref{eq:X} at $t - \frac{N_n}{n}$. In fact, different kind of asymptotics of the MSE, where truncation error becomes dominant, could be derived when  \eqref{eq:gammaasymp} does not hold, under some additional assumptions, but we do not pursue this direction in the present paper. 

While the rate of convergence in \eqref{eq:rate1} is fully determined by the roughness index $\alpha$, which may seem discouraging at first, it turns out that the quantity $J(\alpha,\kappa,\mathbf{b})$, which we shall call the \emph{asymptotic} MSE, can vary a lot, depending on how we choose $\kappa$ and $\mathbf{b}$, and can have a substantial impact on the precision of the approximation of $X$. It is immediate from \eqref{eq:Jdef} that increasing $\kappa$ will decrease $J(\alpha,\kappa,\mathbf{b})$. Moreover, for given $\alpha$ and $\kappa$, it is straightforward to choose $\mathbf{b}$ so that $J(\alpha,\kappa,\mathbf{b})$ is minimized, as shown in the following result.

\begin{proposition}[Optimal discretization]\label{prop:optD}
Let $\alpha \in (-\frac{1}{2},\frac{1}{2})\setminus \{ 0\}$ and $\kappa \geq 0$. Among all sequences $\mathbf{b}=\{b_k\}_{k=\kappa+1}^{\infty}$ with $b_k \in [k-1,k]\setminus \{0 \}$ for $k \geq \kappa +1$, the function $J(\alpha,\kappa,\mathbf{b})$, and consequently the asymptotic MSE induced by the discretization, is minimized by the sequence $\mathbf{b}^*$ given by
\begin{align*}
b_k^* = \bigg( \frac{k^{\alpha+1} - (k-1)^{\alpha+1}}{\alpha+1}\bigg)^{1/\alpha}, \quad k\geq \kappa+1. 
\end{align*}
\end{proposition}

\begin{proof}
Clearly, a sequence $\mathbf{b}=\{b_k\}_{k=\kappa+1}^{\infty}$ minimizes the function $J(\alpha,\kappa,\mathbf{b})$ if and only if $b_k$ minimizes $\int_{k-1}^k (y^\alpha-b_k^\alpha)^2 dy$ for any $k \geq \kappa+1$. By standard $L^2$-space theory, $c \in \R$ minimizes the integral $\int_{k-1}^k (y^\alpha -c)^2 dy$ if and only if the function $y \mapsto y^\alpha - c$ is orthogonal in $L^2$ to all constant functions. This is tantamount to
\begin{equation*}
\int_{k-1}^k (y^\alpha - c) dy = 0,
\end{equation*}
and computing the integral and solving for $c$ yields
\begin{equation*}
c = \frac{k^{\alpha+1}- (k-1)^{\alpha+1}}{\alpha + 1}.
\end{equation*}
Setting $b^{*}_k := c^{1/\alpha} \in (k-1,k)$ completes the proof.
\end{proof}

To understand how much increasing $\kappa$ and using the optimal sequence $\mathbf{b}^*$ from Proposition \ref{prop:optD} improves the approximation, we study numerically the asymptotic \emph{root mean square error} (RMSE) $\sqrt{J(\alpha,\kappa,\mathbf{b})}$. In particular, we assess how much the asymptotic RMSE decreases relative to RMSE of the forward Riemann-sum scheme ($\kappa=0$ and $\mathbf{b} = \mathbf{b}_{\mathrm{FWD}}$) using the quantity
\begin{equation}\label{eq:reduction}
\textrm{reduction in asymptotic RMSE} = - \frac{\sqrt{J(\alpha,\kappa,\mathbf{b})}-\sqrt{J(\alpha,0,\mathbf{b}_{\mathrm{FWD}})}}{\sqrt{J(\alpha,0,\mathbf{b}_{\mathrm{FWD}})}} \cdot 100\%.
\end{equation}

The results are presented in Figure \ref{fig:optimal}. We find that employing the hybrid scheme with $\kappa \geq 1$ leads to a substantial reduction in the asymptotic RMSE relative to the forward Riemann-sum scheme when $\alpha \in (-\frac{1}{2},0)$. Indeed, when $\kappa \geq 1$, the asymptotic RMSE, as a function of $\alpha$, does not blow up as $\alpha \rightarrow -\frac{1}{2}$, while with $\kappa = 0$ it does. This explains why the reduction in the asymptotic RMSE approaches $100\%$ as as $\alpha \rightarrow -\frac{1}{2}$.  When $\alpha \in (0,\frac{1}{2})$, the improvement achieved using the hybrid scheme is more modest, but still considerable. Figure \ref{fig:optimal} also highlights the importance of using the optimal sequence $\mathbf{b}^*$, instead of $\mathbf{b}_{\mathrm{FWD}}$, as evaluation points in the scheme, in particular when $\alpha \in (0,\frac{1}{2})$. Finally, we observe that increasing $\kappa$ beyond $2$ does not appear to lead to a significant further reduction. Indeed, in our numerical experiments, reported in Section \ref{sec:estimate_a} and \ref{sec:roughvol} below, we observe that using $\kappa = 1,2$ already leads to good results.

\begin{figure}[!t] 
\centering
\begin{tabular}{rl} 
\includegraphics[scale=0.95,trim=0.5cm 1cm 0.7cm 1cm, clip=TRUE]{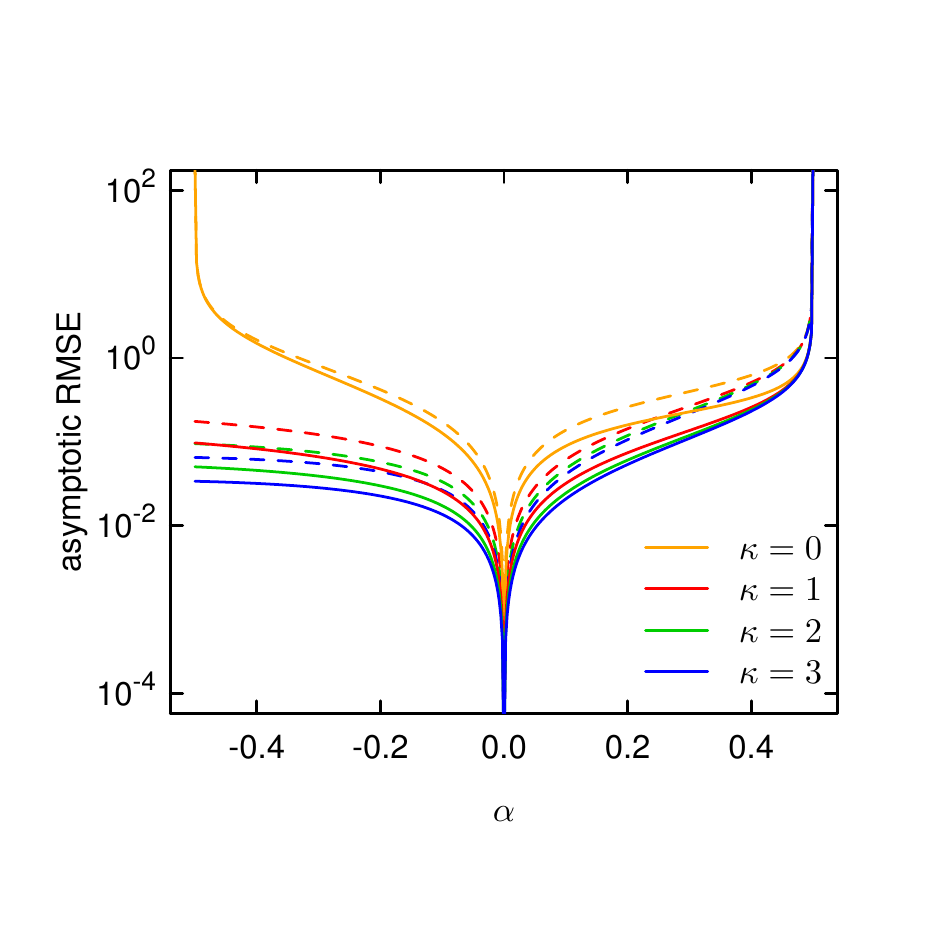} & \includegraphics[scale=0.95,trim=0.5cm 1cm 0.7cm 1cm, clip=TRUE]{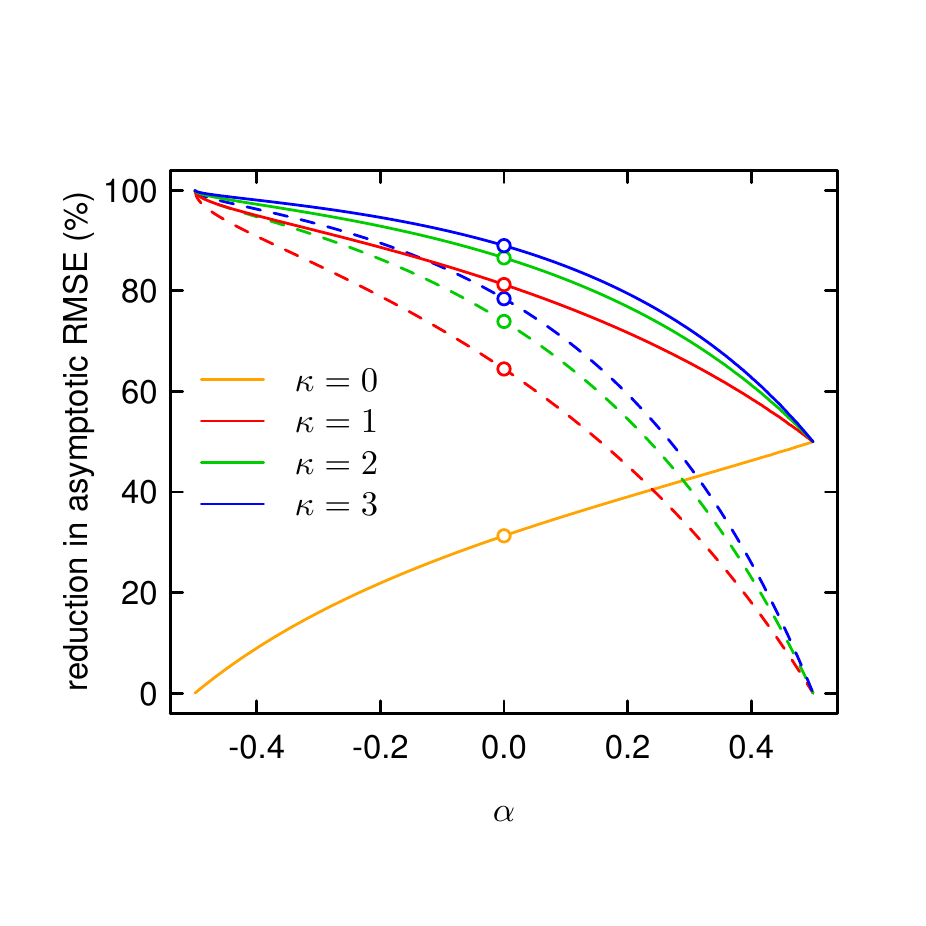}
\end{tabular}
\caption{\it Left: The asymptotic RMSE given by $\sqrt{J(\alpha,\kappa,\mathbf{b})}$ as a function of $\alpha \in (-\frac{1}{2},\frac{1}{2})\setminus \{0\}$ for $\kappa = 0,1,2,3$ using $\mathbf{b} = \mathbf{b}^*$ of Proposition \ref{prop:optD} (solid lines) and $\mathbf{b}=\mathbf{b}_{\mathrm{FWD}}$ (dashed lines).
Right: Reduction in the asymptotic RMSE relative to the forward Riemann-sum scheme ($\kappa=0$ and $\mathbf{b}=\mathbf{b}_{\mathrm{FWD}}$) given by the formula \eqref{eq:reduction}, plotted as a function of $\alpha \in (-\frac{1}{2},\frac{1}{2})\setminus \{0\}$ for $\kappa = 0,1,2,3$ using $\mathbf{b} = \mathbf{b}^*$  (solid lines) and for $\kappa = 1,2,3$ using $\mathbf{b}=\mathbf{b}_{\mathrm{FWD}}$ (dashed lines).
 In all computations, we have used the approximations outlined in Remark \ref{rem:Japprox} with $N=1\,000\,000$.}
\label{fig:optimal}
\end{figure}

\begin{remark}\label{rem:Japprox}
It is non-trivial to evaluate the quantity $J(\alpha,\kappa,\mathbf{b})$ numerically. Computing the integral in \eqref{eq:Jdef} explicitly, we can approximate $J(\alpha,\kappa,\mathbf{b})$ by
\begin{equation*}
J_N(\alpha,\kappa,\mathbf{b}) := \sum_{k=\kappa+1}^N \bigg( \frac{k^{2\alpha+1}-(k-1)^{2\alpha+1}}{2\alpha +1}-\frac{2 b^\alpha_k\big(k^{\alpha+1}-(k-1)^{\alpha+1}\big)}{\alpha +1} + b^{2\alpha}_k \bigg)
\end{equation*}
with some large $N\in \N$. This approximation is adequate when $\alpha \in (-\frac{1}{2},0)$, but its accuracy deteriorates when $\alpha \rightarrow \frac{1}{2}$. In particular, the singularity of the function $\alpha \mapsto J(\alpha,\kappa,\mathbf{b})$ at $\frac{1}{2}$ is difficult to capture using $J_N(\alpha,\kappa,\mathbf{b})$ with numerically feasible values of $N$. To overcome this numerical problem, we introduce a correction term in the case $\alpha \in (0,\frac{1}{2})$. The correction term can be derived informally as follows. By the mean value theorem, and since $b^*_k \approx k-\frac{1}{2}$ for large $k$, we have
\begin{equation*}
(y^\alpha - b^\alpha_k)^2 = \alpha^2 \xi^{2\alpha -2 }(y - b_k)^2 \approx \begin{cases}
\alpha^2 k^{2\alpha -2 }(y - k)^2, & \mathbf{b} = \mathbf{b}_{\mathrm{FWD}},\\
\alpha^2 k^{2\alpha -2 }(y - k + \frac{1}{2})^2, & \mathbf{b} = \mathbf{b}^*,
\end{cases}
\end{equation*}
where $\xi = \xi(y,b_k) \in [k-1,k]$, for large $k$. Thus, for large $N$, we obtain
\begin{equation*}
\begin{split}
J(\alpha,\kappa,\mathbf{b}) - J_N(\alpha,\kappa,\mathbf{b}) &= \sum_{k=N+1}^\infty \int_{k-1}^k (y^{\alpha} - b_k^{\alpha})^2dy \\
& \approx \begin{cases}
\alpha^2 \sum_{k=N+1}^\infty k^{2\alpha -2 } \int_{k-1}^k (y - k)^2 dy, & \mathbf{b} = \mathbf{b}_{\mathrm{FWD}},\\
\alpha^2  \sum_{k=N+1}^\infty k^{2\alpha -2 }\int_{k-1}^k (y - k + \frac{1}{2})^2dy, & \mathbf{b} = \mathbf{b}^*,
\end{cases} \\
& = \begin{cases}
\frac{\alpha^2}{3} \zeta(2-2\alpha,N+1), & \mathbf{b} = \mathbf{b}_{\mathrm{FWD}},\\
\frac{\alpha^2}{12} \zeta(2-2\alpha,N+1), & \mathbf{b} = \mathbf{b}^*,
\end{cases}
\end{split}
\end{equation*}
where $\zeta(x,s) := \sum_{k=0}^\infty \frac{1}{(k+s)^x}$, $x>1$, $s > 0$, is the \emph{Hurwitz zeta function}, which can be evaluated using accurate numerical algorithms.
\end{remark} 

\begin{remark}
Unlike the Fourier-based method of Benth et al.\ \cite{heidar13fourier}, the hybrid scheme does not require truncating the singularity of the kernel function $g$ when $\alpha \in (-\frac{1}{2},0)$, which is beneficial to maintaining the accuracy of the scheme when $\alpha$ is near $-\frac{1}{2}$.
Let us briefly analyze the effect of truncating the singularity of $g$ on the approximation error, cf.\ \cite[pp.\ 75--76]{heidar13fourier}. 
Consider, for any $\varepsilon>0$, the modified $\mathcal{BSS}$ process
\begin{equation*}
\tilde{X}_\varepsilon(t) := \int_{-\infty}^t g_\varepsilon(t-s) \sigma(s) dW(s), \quad t \in \R,
\end{equation*}
defined using the truncated kernel function
\begin{equation*}
g_\varepsilon(x) := \begin{cases}
g(\varepsilon), & x \in (0,\varepsilon],\\
g(x), & x \in (\varepsilon,\infty).
\end{cases}
\end{equation*}
Adapting the proof of Theorem \ref{th:mainTh} in a straightforward manner, it is possible to show that, under \eqref{ass:zero} and \eqref{ass:infb},
\begin{multline*}
\mathbb{E}\big[\big|X(t)-\tilde{X}_\varepsilon(t)\big|^2\big]  = \E[\sigma(0)^2] \int_{0}^\varepsilon \big(g(s) - g(\varepsilon)\big)^2 ds
  \\ \sim  \underbrace{\bigg( \frac{1}{2\alpha +1}-\frac{2}{\alpha+1}+1\bigg)}_{=:\tilde{J}(\alpha)} \E[\sigma(0)^2] \varepsilon^{2\alpha+1} L_g(\varepsilon)^2, \quad \varepsilon \downarrow 0,
\end{multline*}
for any $t \in \R$. While the rate of convergence, as $\varepsilon \downarrow 0$, of the MSE that arises from replacing $g$ with $g_\varepsilon$ is analogous to the rate of convergence of the hybrid scheme, it is important to note that the factor $\tilde{J}(\alpha)$ blows up as $\alpha \downarrow -\frac{1}{2}$. In fact, $\tilde{J}(\alpha)$ is equal to the first term in the series that defines $J(\alpha,0,\mathbf{b}_{\mathrm{FWD}})$ and
\begin{equation*}
\tilde{J}(\alpha) \sim J(\alpha,0,\mathbf{b}_{\mathrm{FWD}}), \quad \alpha \downarrow -\frac{1}{2}, 
\end{equation*}
which indicates that the effect of truncating the singularity, in terms of MSE, is similar to the effect of using the forward Riemann-sum scheme to discretize the process when $\alpha$ is near $-\frac{1}{2}$. In particular, the truncation threshold $\varepsilon$ would then have to be very small in order to keep the truncation error in check.
\end{remark}

\subsection{Extension to truncated Brownian semistationary processes}\label{sec:tbss}

It is useful to extend the hybrid scheme to a class of non-stationary processes that are closely related to $\mathcal{BSS}$ processes. This extension is important in connection with an application to the so-called rough Bergomi model, which we discuss in Section \ref{sec:roughvol}, below. More precisely, we consider processes of the form
\begin{equation}\label{eq:Y}
Y(t) = \int_0^t g(t-s) \sigma(s) dW(s), \quad t \geq 0,
\end{equation}
where the kernel function $g$, volatility process $\sigma$ and driving Brownian motion $W$ are as before. We call $Y$ a \emph{truncated Brownian semistationary} ($\mathcal{TBSS}$) process, as $Y$ is obtained from the $\mathcal{BSS}$ process $X$ by truncating the stochastic integral in $\eqref{eq:X}$ at $0$.
Of the preceding assumptions, only \eqref{ass:zero} and \eqref{ass:deriv} are needed to ensure that the stochastic integral in \eqref{eq:Y} exists --- in fact, of \eqref{ass:deriv}, only the requirement that $g$ is differentiable on $(0,\infty)$ comes into play.

The $\mathcal{TBSS}$ process $Y$ does not have covariance stationary increments, so we define its (time-dependent) variogram as
\begin{equation*}
V_Y(h,t) := \E[|Y(t+h)-Y(t)|^2], \quad h,t \geq 0.
\end{equation*}
Extending Proposition \ref{prop:localbehavior}, we can describe the behavior of $h \mapsto V_Y(h,t)$ near zero as follows. The existence of a locally H\"older continuous modification is then a straightforward consequence.
We omit the proof of this result, as it would be straightforward adaptation of the proof of Proposition \ref{prop:localbehavior}.

\begin{proposition}[Local behavior and continuity]
Suppose that \eqref{ass:zero} and \eqref{ass:deriv} hold.
\begin{enumerate}[label=(\roman*),ref=\roman*,leftmargin=*]
\item The variogram of $Y$ satisfies for any $t \geq 0$,
\begin{equation*}
V_Y(h,t) \sim \E[\sigma(0)^2] \bigg(\frac{1}{2\alpha + 1} + \mathbf{1}_{(0,\infty)}(t)\int_0^\infty \big( (y+1)^\alpha - y^\alpha\big)^2 dy\bigg) h^{2\alpha+1} L_g(h)^2, \quad h \rightarrow 0,
\end{equation*}
which implies that $h \mapsto V_Y(h,t)$ is regularly varying at zero with index $2\alpha+1$.
\item The process $Y$ has a modification with locally $\phi$-H\"older continuous trajectories for any $\phi \in (0,\alpha + \frac{1}{2})$.
\end{enumerate}
\end{proposition}

Note that while the increments of $Y$ are not covariance stationary, the asymptotic behavior of $V_Y(h,t)$ is the same as that of $V_X(h)$ as $h \rightarrow 0$ (cf.\ Proposition \ref{prop:localbehavior}) for any $t>0$. Thus, the increments of $Y$ (apart from increments starting at time $0$) are locally like the increments of $X$.

We define the hybrid scheme to discretize $Y(t)$, for any $t \geq 0$, as
\begin{equation}\label{eq:Yhybrid}
Y_n(t) := \check{Y}_n(t) + \hat{Y}_n(t),
\end{equation}
where
\begin{align*}
\check{Y}_n(t) & := \sum_{k=1}^{\min\{\lfloor nt \rfloor, \kappa \}} L_g\bigg( \frac{k}{n} \bigg) \sigma\bigg( t-\frac{k}{n}\bigg) \int_{t-\frac{k}{n}}^{t-\frac{k}{n}+\frac{1}{n}} (t-s)^\alpha dW(s), \\
\hat{Y}_n(t) & := \sum_{k=\kappa+1}^{\lfloor nt \rfloor} g\bigg(\frac{b_k}{n}\bigg)\sigma\bigg(t- \frac{k}{n}\bigg)\Bigg( W\bigg(t-\frac{k}{n}+\frac{1}{n}\bigg) - W\bigg(t-\frac{k}{n}\bigg)\Bigg).
\end{align*}
In effect, we simply drop the summands in \eqref{eq:checkX} and \eqref{eq:hatX} that correspond to integrals and increments on the negative real line. We make remarks on the implementation of this scheme in Section \ref{ssec:impl}, below.

The MSE of hybrid scheme for the $\mathcal{TBSS}$ process $Y$ has the following asymptotic behavior as $n \rightarrow \infty$, which is, in fact, identical to the asymptotic behavior of the MSE of the hybrid scheme for $\mathcal{BSS}$ processes. We omit the proof of this result, which would be a simple modification of the proof of Theorem \ref{th:mainTh}.

\begin{theorem}[Asymptotics of mean square error]\label{th:mainTh2}
Suppose that \eqref{ass:zero} and \eqref{ass:deriv} hold, and that for some $\delta>0$,
\begin{align*}
\mathbb{E}[|\sigma(s) - \sigma(0)|^2] = \mathcal{O}\big(s^{2\alpha +1+\delta}\big),  \quad s\downarrow 0.
\end{align*}
Then for all $t>0$,
\begin{align*}\label{eq:rate1}
\mathbb{E}[|Y(t) - Y_n(t)|^2] \sim J(\alpha,\kappa,\mathbf{b}) \mathbb{E}[\sigma(0)^2] n^{-(2\alpha +1)} L_g(1/n)^2 , \quad n\rightarrow \infty,
\end{align*}
where $J(\alpha,\kappa,\mathbf{b})$ is as in Theorem \ref{th:mainTh}
\end{theorem}

\begin{remark}
Under the assumptions of Theorem \ref{th:mainTh2}, the conclusion of Corollary \ref{cor:covariance} holds mutatis mutandis. In particular, the covariance structure of the discretized $\mathcal{TBSS}$ process approaches that of $Y$ when $n \rightarrow \infty$.
\end{remark}

\section{Implementation and numerical experiments}\label{sec:app}

\subsection{Practical implementation}\label{ssec:impl}

Simulating the $\mathcal{BSS}$ process $X$ on the equidistant grid $\{0,\frac{1}{n},\frac{2}{n},\ldots,\frac{\lfloor nT \rfloor}{n} \}$ for some $T>0$ using the hybrid scheme entails generating
\begin{equation}\label{eq:hybridprac}
X_n\bigg(\frac{i}{n} \bigg), \quad i = 0,1,\ldots,\lfloor nT \rfloor.
\end{equation}
Provided that we can simulate the random variables
\begin{align}
W^n_{i,j} & := \int_{\frac{i}{n}}^{\frac{i+1}{n}} \bigg(\frac{i+j}{n}-s\bigg)^\alpha dW(s), & i & = -N_n, -N_n+1,\ldots,\lfloor nT \rfloor-1, \quad j = 1,\ldots,\kappa,\label{eq:hybridinputW1} \\
W^n_i & := \int_{\frac{i}{n}}^{\frac{i+1}{n}} dW(s), & i & = -N_n, -N_n+1,\ldots,\lfloor nT \rfloor-1,\label{eq:hybridinputW2}\\
\sigma^n_i & := \sigma\bigg( \frac{i}{n} \bigg), & i & = -N_n, -N_n+1,\ldots,\lfloor nT \rfloor-1, \nonumber
\end{align}
we can compute \eqref{eq:hybridprac} via the formula
\begin{equation}\label{eq:hybridkey}
X_n\bigg(\frac{i}{n}\bigg) = \underbrace{\sum_{k=1}^\kappa L_g\bigg( \frac{k}{n}\bigg) \sigma^n_{i-k} W^n_{i-k,k}}_{=\check{X}_n(\frac{i}{n})} + \underbrace{\sum_{k=\kappa+1}^{N_n} g\bigg( \frac{b^*_k}{n}\bigg) \sigma^n_{i-k} W^n_{i-k}}_{=\hat{X}_n(\frac{i}{n})}.
\end{equation}
In order to simulate \eqref{eq:hybridinputW1} and \eqref{eq:hybridinputW2}, it is instrumental to note that the $\kappa+1$-dimensional random vectors
\begin{equation*}
\mathbf{W}^n_i := \big(W^n_i,
W^n_{i,1},\ldots,
W^n_{i,\kappa}\big), \quad i = -N_n, -N_n+1,\ldots,\lfloor nT \rfloor-1,
\end{equation*}
are i.i.d.\ according to a multivariate Gaussian distribution with mean zero and covariance matrix $\Sigma$ given by
\begin{align*}
\Sigma_{1,1} & = \frac{1}{n}, & \Sigma_{1,j} = \Sigma_{j,1} & = \frac{(j-1)^{\alpha+1}-(j-2)^{\alpha+1}}{(\alpha+1)n^{\alpha+1}}, &
\Sigma_{j,j} & = \frac{(j-1)^{2\alpha+1}-(j-2)^{2\alpha+1}}{(2\alpha+1)n^{2\alpha+1}}, 
\end{align*}
for $j = 2,\ldots,\kappa+1$, and
\begin{multline}\label{eq:Wiener-cov}
 \Sigma_{j,k}  =  \frac{1}{(\alpha+1) n^{2\alpha+1}}\Bigg((j-1)^{\alpha+1}(k-1)^\alpha {}_2 F_1\bigg(-\alpha,1,\alpha+2,\frac{j-1}{k-1}\bigg) \\
 - (j-2)^{\alpha+1}(k-2)^{\alpha} {}_2 F_1\bigg(-\alpha,1,\alpha+2,\frac{j-2}{k-2}\bigg)\Bigg),
\end{multline}
for $j$,\ $k = 2,\ldots,\kappa+1$ such that $j < k$, where ${}_2 F_1$ stands for the \emph{Gauss hypergeometric function}, see, e.g., \cite[p.\ 56]{EMOT1953} for the definition. (When $k < j$, set $\Sigma_{j,k} = \Sigma_{k,j}$.) For the convenience of the reader, we provide a proof of \eqref{eq:Wiener-cov} in Section \ref{ssec:hyper}.

Thus, $\{\mathbf{W}^n_i\}_{i=-N_n}^{\lfloor nT \rfloor-1}$ can be generated by taking independent draws from the multivariate Gaussian distribution $N_{\kappa+1}(\mathbf{0},\Sigma)$. If the volatility process $\sigma$ is independent of $W$, then $\{\sigma^n_i\}_{i=-N_n}^{\lfloor nT \rfloor-1}$ can be generated separately, possibly using exact methods. (Exact methods are available, e.g., for Gaussian processes, as mentioned in the introduction, and diffusions, see \cite{BeskosRoberts05}.) In the case where $\sigma$ depends on $W$, simulating $\{\mathbf{W}^n_i\}_{i=-N_n}^{\lfloor nT \rfloor-1}$ and $\{\sigma^n_i\}_{i=-N_n}^{\lfloor nT \rfloor-1}$ is less straightforward. That said, if $\sigma$ is driven by a standard Brownian motion $Z$, correlated with $W$, say, one could rely on a factor decomposition
\begin{equation}\label{eq:factorBm}
Z(t) := \rho W(t) + \sqrt{1-\rho^2} W_\perp(t), \quad t \in \R,
\end{equation}
where $\rho \in [-1,1]$ is the correlation parameter and $\{W_\perp(t)\}_{t \in [0,T]}$ is a standard Brownian motion independent of $W$. Then one would first generate $\{\mathbf{W}^n_i\}_{i=-N_n}^{\lfloor nT \rfloor-1}$, use \eqref{eq:factorBm} to generate $\{ Z(\frac{i+1}{n})-Z(\frac{i}{n}) \}_{i=-N_n}^{\lfloor nT \rfloor-1}$ and employ some appropriate approximate method to produce $\{\sigma^n_i\}_{i=-N_n}^{\lfloor nT \rfloor-1}$ thereafter. This approach has, however, the caveat that it induces an additional approximation error, not quantified in Theorem \ref{th:mainTh}.

\begin{remark} In the case of the $\mathcal{TBSS}$ process $Y$, introduced in Section \ref{sec:tbss}, the observations $Y_n( \frac{i}{n})$, $i = 0,1,\ldots,\lfloor nT \rfloor$, given by the hybrid scheme \eqref{eq:Yhybrid} can be computed via
\begin{equation}\label{eq:Yhybridkey}
Y_n\bigg(\frac{i}{n}\bigg) = \sum_{k=1}^{\min\{ i, \kappa \}} L_g\bigg( \frac{k}{n}\bigg) \sigma^n_{i-k} W^n_{i-k,k} + \sum_{k=\kappa+1}^{i} g\bigg( \frac{b^*_k}{n}\bigg) \sigma^n_{i-k} W^n_{i-k},
\end{equation}
using the random vectors $\{\mathbf{W}^n_i\}_{i=0}^{\lfloor nT \rfloor-1}$ and random variables $\{\sigma^n_i\}_{i=0}^{\lfloor nT \rfloor-1}$.
\end{remark}

In the hybrid scheme, it typically suffices to take $\kappa$ to be at most $3$. Thus, in \eqref{eq:hybridkey}, the first sum $\check{X}_{n}(\frac{i}{n})$ requires only a negligible computational effort. By contrast, the number of terms in the second sum $\hat{X}_{n}(\frac{i}{n})$ increases as $n \rightarrow \infty$. It is then useful to note that
\begin{equation*}
\hat{X}_n \bigg( \frac{i}{n} \bigg) = \sum_{k=1}^{N_n} \Gamma_k \Xi_{i-k} = (\Gamma \star \Xi)_i,
\end{equation*}
where
\begin{align*}
\Gamma_k & := \begin{cases}
0, & k = 1,\ldots,\kappa, \\
g\big( \frac{b^*_k}{n} \big),& k = \kappa + 1,\kappa+2,\ldots,N_n,
\end{cases} \\
\Xi_k & := \sigma^n_k W^n_k, \quad k=-N_n, -N_n+1,\ldots,\lfloor nT \rfloor-1.
\end{align*}
and $\Gamma \star \Xi$ stands for the discrete convolution of the sequences $\Gamma$ and $\Xi$.
 It is well-known that the discrete convolution can be evaluated efficiently using a fast Fourier transform (FFT). The computational complexity of simultaneously evaluating $(\Gamma \star \Xi)_i$ for all $i = 0,1,\ldots,\lfloor nT \rfloor$ using an FFT is $\mathcal{O}(N_n \log N_n)$, see \cite[pp.\ 79--80]{Mallat09}, which under \eqref{ass:truncation} translates to $\mathcal{O}(n^{\gamma+1} \log n)$. The computational complexity of the entire hybrid scheme is then $\mathcal{O}(n^{\gamma+1} \log n)$, provided that $\{\sigma^n_i\}_{i=-N_n}^{\lfloor nT \rfloor-1}$ is generated using a scheme with complexity not exceeding $\mathcal{O}(n^{\gamma+1} \log n)$.
As a comparison, we mention that the complexity of an exact simulation of a stationary Gaussian process using circulant embeddings is $\mathcal{O}(n \log n)$ \cite[p.\ 316]{AsmussenGlynn07}, whereas the complexity of the Cholesky factorization is $\mathcal{O}(n^3)$ \cite[p.\ 312]{AsmussenGlynn07}.

\begin{remark} With $\mathcal{TBSS}$ processes, the computational complexity of the hybrid scheme via \eqref{eq:Yhybridkey} is $\mathcal{O}(n \log n)$.
\end{remark}

Figure \ref{fig:paths} presents examples of trajectories of the $\mathcal{BSS}$ process $X$ using the hybrid scheme with $\kappa = 1, 2$ and $\mathbf{b} = \mathbf{b}^*$. We choose the kernel function $g$ to be the gamma kernel (Example \ref{ex:gamma}) with $\lambda = 1$. We also discretize $X$ using the Riemann-sum scheme, $\kappa = 0$ with $\mathbf{b} \in \{  \mathbf{b}_{\mathrm{FWD}},\mathbf{b}^*\}$ (that is, the forward Riemann-sum scheme and its counterpart with optimally chosen evaluation points).
We can make two observations:
 Firstly, we see how the roughness parameter $\alpha$ controls the regularity properties of the trajectories of $X$ --- as we decrease $\alpha$, the trajectories of $X$ become increasingly rough. Secondly, and more importantly, we see how the simulated trajectories coming from the Riemann-sum and hybrid schemes can be rather different, even though we use the same innovations for the driving Brownian motion. In fact, the two variants of the hybrid scheme ($\kappa =  1, 2$) yield almost identical trajectories, while the Riemann-sum scheme ($\kappa = 0$) produces trajectories that are comparatively smoother, this difference becoming more apparent as $\alpha$ approaches $-\frac{1}{2}$. Indeed, in the extreme case with $\alpha = -0.499$, both variants of the Riemann-sum scheme break down and yield anomalous trajectories with very little variation, while the hybrid scheme continues to produce accurate results.
The fact that the hybrid scheme is able to reproduce the fine properties of rough $\BSS$ processes, even for values of $\alpha$ very close to $-\frac{1}{2}$, is backed up by a further experiment reported in the following section.

\begin{figure}[tbp] 
\centering 
\includegraphics[scale=0.88]{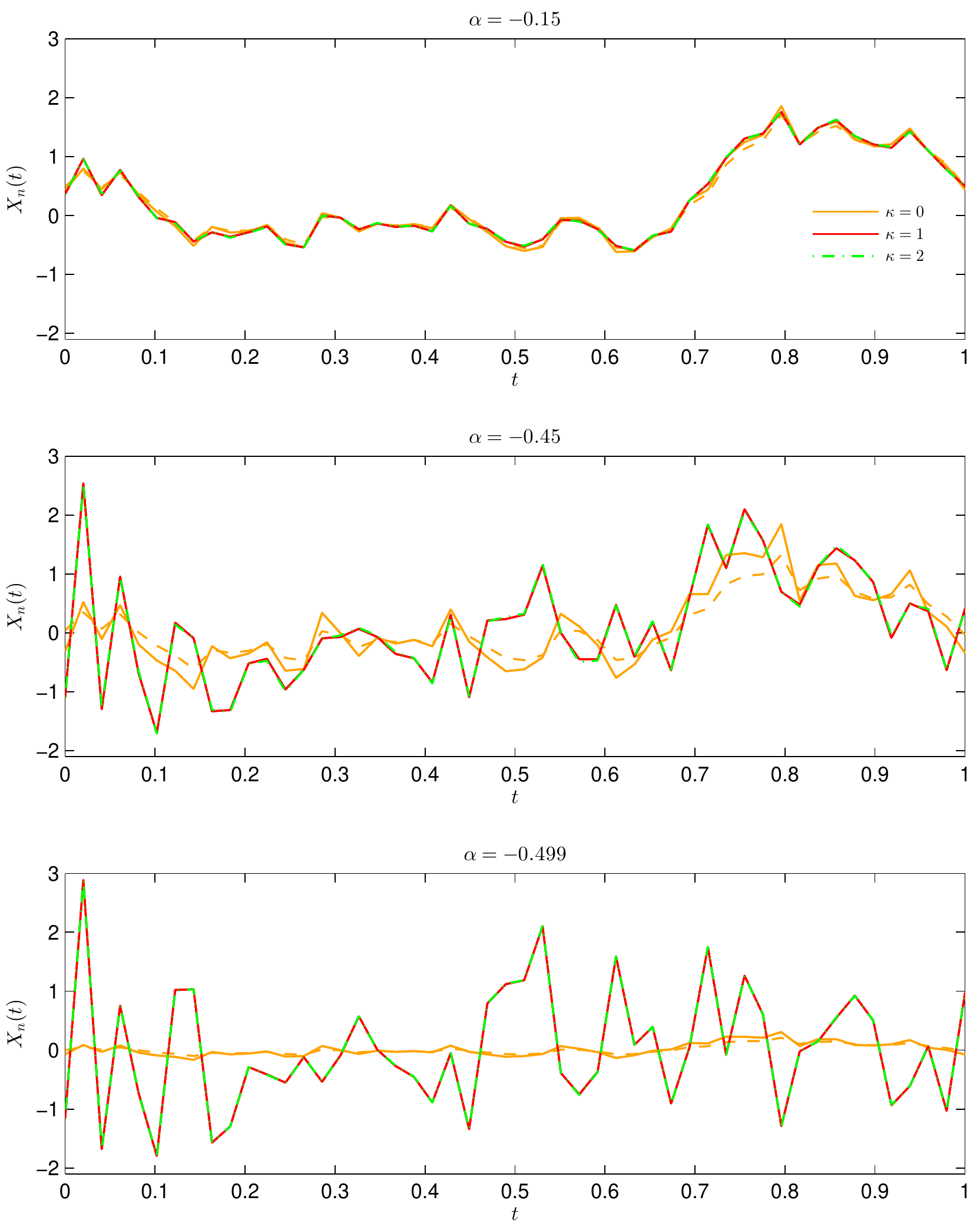} 
\caption{\it Discretized trajectories of a $\mathcal{BSS}$ process, where $g$ is the gamma kernel (Example \ref{ex:gamma}), $\lambda = 1$ and $\sigma(t) = 1$ for all $t\in \R$. Trajectories consisting of $n=50$ observations on $[0,1]$ were generated with the hybrid scheme ($\kappa = 1, 2$ and $\mathbf{b} = \mathbf{b}^*$) and Riemann-sum scheme ($\kappa=0$ and $\mathbf{b} = \mathbf{b}^*$ (solid lines), $\mathbf{b}=\mathbf{b}_{\mathrm{FWD}}$ (dashed lines)), using the same innovations for the driving Brownian motion in all cases and $N_n =\lfloor 50^{1.5}\rfloor = 353$. 
The simulated processes were normalized to have unit (stationary) variance.\label{fig:paths}}
\end{figure}

\subsection{Estimation of the roughness parameter}\label{sec:estimate_a}

Suppose that we have observations $X(\frac{i}{m})$, $i = 0,1,\ldots,m$, of the $\BSS$ process $X$, given by \eqref{eq:X}, for some $m \in \N$. Barndorff-Nielsen et al.\ \cite{ole_jose_mark13} and Corcuera et al.\ \cite{jose_emil_mikko_mark13} discuss how the roughness index $\alpha$ can be estimated consistently as $m \rightarrow \infty$. The method is based on the \emph{change-of-frequency} (COF) statistics
\begin{align*}
\mathrm{COF}(X,m) = \frac{\sum_{k=5}^m \big| X\big(\frac{k}{m}\big) - 2X\big(\frac{k-2}{m}\big) + X\big(\frac{k-4}{m}\big) \big|^2}{  \sum_{k=3}^m \big| X\big(\frac{k}{m}\big) - 2X\big(\frac{k-1}{m}\big) + X\big(\frac{k-2}{m}\big) \big|^2}, \quad m\geq 5,
\end{align*}
which compare the realized quadratic variations of $X$, using second-order increments, with two different lag lengths. Corcuera et al.\ \cite{jose_emil_mikko_mark13} have shown that under some assumptions on the process $X$, which are similar to \eqref{ass:zero}, \eqref{ass:deriv} and \eqref{ass:infb} albeit slightly more restrictive, it holds that \begin{align}\label{eq:estimator}
\hat{\alpha}(X,m) := \frac{\log \big(\textnormal{COF}(X,m)\big)}{2 \log 2} - \frac{1}{2} \stackrel{\mathbb{P}}{\longrightarrow} \alpha, \quad m\rightarrow \infty.
\end{align}
An in-depth study of the finite sample performance of this COF estimator can be found in \cite{bennedsen14a}.

\begin{figure}[tbp] 
\centering 
\includegraphics[scale=0.93]{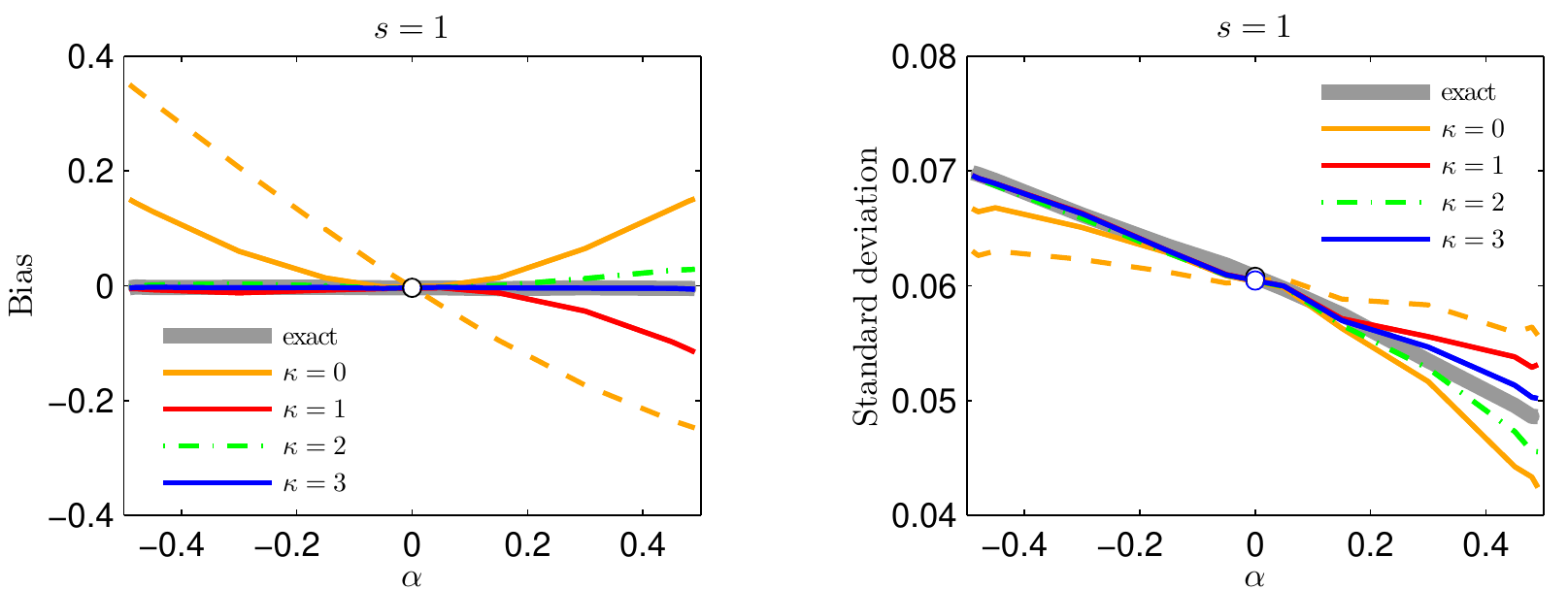} \\
\includegraphics[scale=0.93]{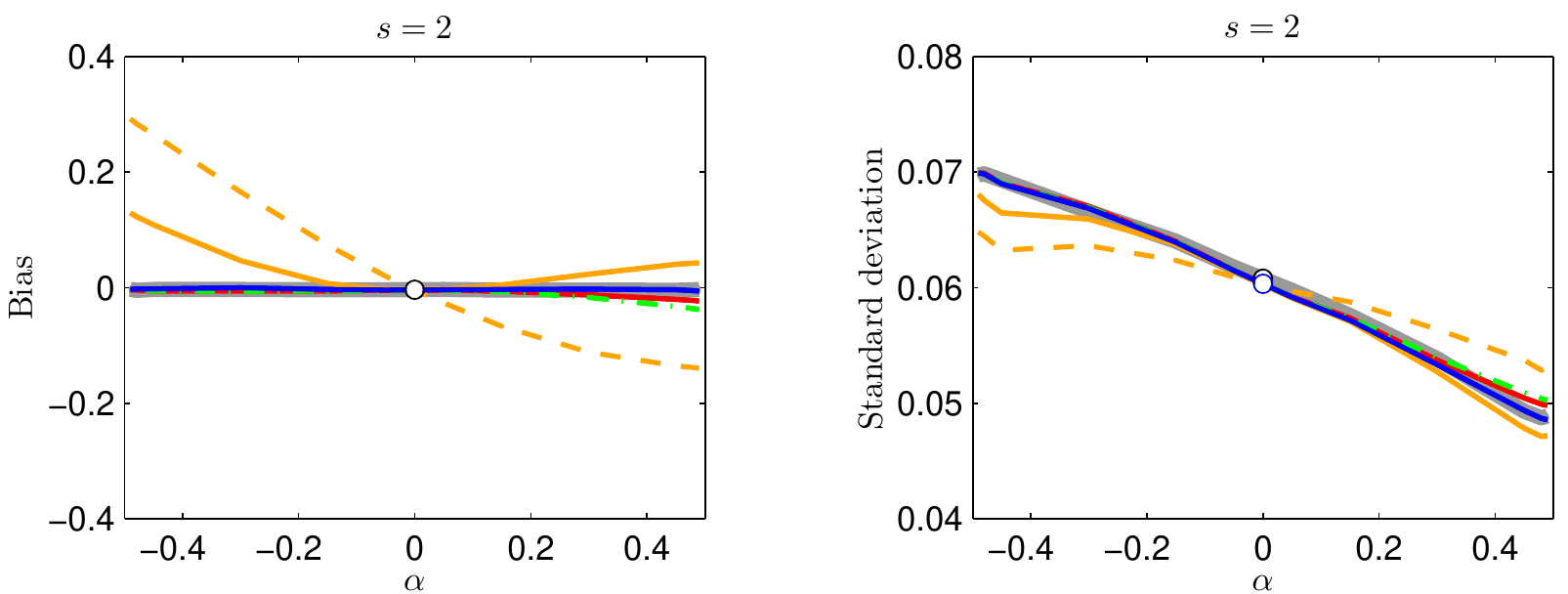} \\
\includegraphics[scale=0.93]{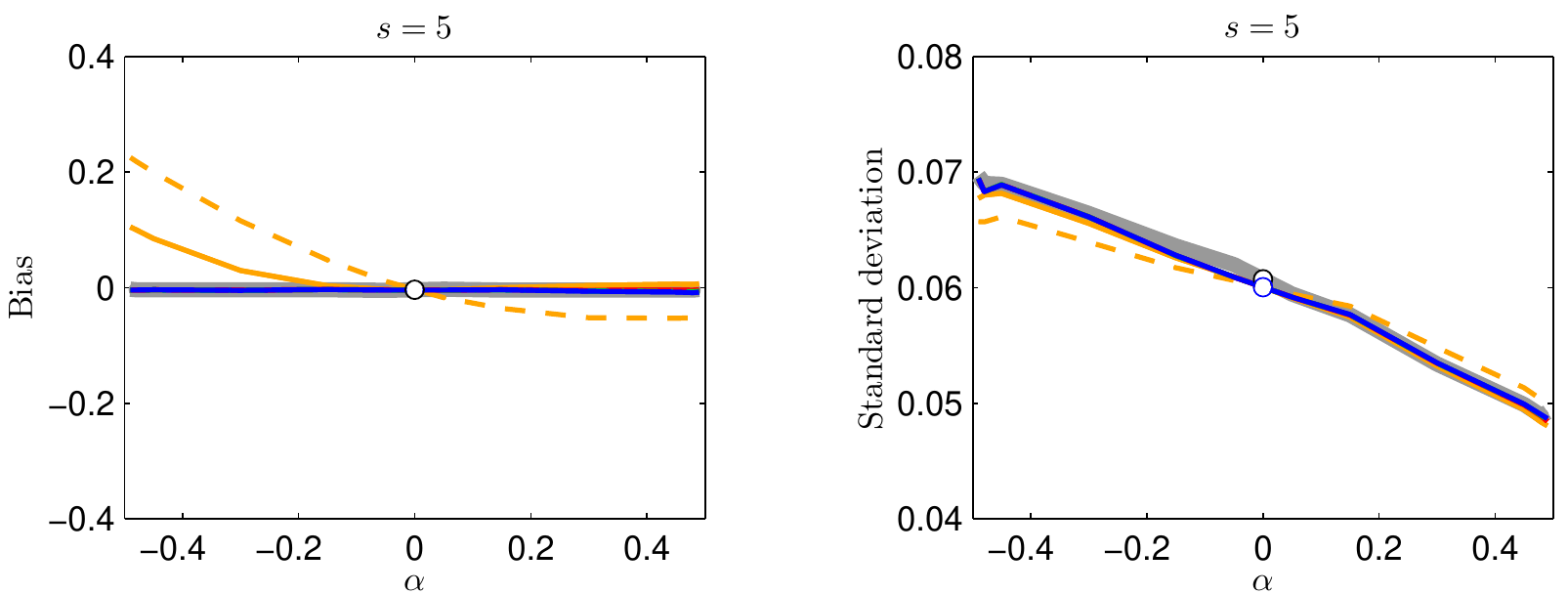}
\caption{\it Bias and standard deviation of the COF estimator \eqref{eq:estimator} of the roughness index $\alpha$, when applied to discretized trajectories of a $\BSS$ process with the gamma kernel (Example \ref{ex:gamma}), $\lambda = 1$ and $\sigma(t) = 1$ for all $t\in \R$. Trajectories were generated using an exact method based on the Cholesky factorization, the hybrid scheme ($\kappa = 1, 2, 3$ and $\mathbf{b} = \mathbf{b}^*$) and Riemann-sum scheme ($\kappa=0$ and $\mathbf{b} = \mathbf{b}^*$ (solid lines), $\mathbf{b}=\mathbf{b}_{\mathrm{FWD}}$ (dashed lines)). In the experiment, $n = ms$ observations were generated, where $m=500$ and $s \in \{1, 2, 5\}$, on $[0,1]$ using $N_n= \lfloor n^{1.5} \rfloor$.
 Every $s$-th observation was then subsampled, resulting in $m=500$ observations that were used to compute the estimate $\hat{\alpha}(X_n,m)$ of the roughness index $\alpha$. Number of Monte Carlo replications: $10\,000$.}
\label{fig:a_error}
\end{figure}

To examine how well the hybrid scheme reproduces the fine properties of the $\BSS$ process in terms of regularity/roughness, we 
apply the COF estimator to discretized trajectories of $X$, where the kernel function $g$ is again the gamma kernel (Example \ref{ex:gamma}) with $\lambda = 1$, generated using the hybrid scheme with $\kappa = 1,2,3$ and $\mathbf{b} = \mathbf{b}^*$. We consider the case where the volatility process satisfies $\sigma(t) = 1$, that is, the process $X$ is Gaussian. This allows us to quantify and control for the intrinsic bias and noisiness, measured in terms of standard deviation, of the estimation method itself, by initially applying the estimator to trajectories that have been simulated using an exact method based on the Cholesky factorization. We then study the behavior of the estimator when applied to a discretized trajectory, while decreasing the step size of the discretization scheme.
More precisely, we simulate $\hat{\alpha}(X_n,m)$, where $m=500$ and $X_n$ is the hybrid scheme for $X$ with $n = ms$ and $s \in \{1,2,5\}$. This means that we compute $\hat{\alpha}(X_n,m)$ using $m$ observations obtained by subsampling every $s$-th observation in the sequence $X_n(\frac{i}{n})$, $i = 0,1,\ldots,n$. 
As a comparison, we repeat these simulations substituting the hybrid scheme with the Riemann-sum scheme, using $\kappa = 0$ with $\mathbf{b}\in \{\mathbf{b}_{\mathrm{FWD}},\mathbf{b}^*\}$. 

The results are presented in Figure \ref{fig:a_error}. We observe that the intrinsic bias of the estimator with $m = 500$ observations is negligible and hence the bias of the estimates computed from discretized trajectories is then attributable to approximation error arising from the respective discretization scheme, where positive (resp.\ negative) bias indicates that the simulated trajectories are smoother (resp.\ rougher) than those of the process $X$. Concentrating first on the baseline case $s = 1$, we note that the hybrid scheme produces essentially unbiased results when $\alpha \in (-\frac{1}{2},0)$, while there is moderate bias when $\alpha \in (0,\frac{1}{2})$, which disappears when passing from $\kappa=1$ to $\kappa=3$, even for values of $\alpha$ very close to $\frac{1}{2}$. (The largest value of $\alpha$ considered in our simulations is $\alpha = 0.49$; one would expect the performance to weaken as $\alpha$ approaches $\frac{1}{2}$, cf.\ Figure \ref{fig:optimal}, but this range of parameter values seems to be of limited practical interest.) The standard deviations exhibit a similar pattern.
The corresponding results for the Riemann-sum scheme are clearly inferior, exhibiting significant bias, while using optimal evaluation points ($\mathbf{b} = \mathbf{b}^*$) improves the situation slightly. In particular, the bias in the case $\alpha \in (-\frac{1}{2},0)$ is positive, indicating too smooth discretized trajectories, which is connected with the failure of the Riemann-sum scheme with $\alpha$ near $-\frac{1}{2}$, illustrated in Figure \ref{fig:paths}.
With $s=2$ and $s=5$, the results improve with both schemes. Notably, in the case $s=5$, the performance of the hybrid scheme even with $\kappa = 1$ is on a par with the exact method. However, the improvements with the Riemann-sum scheme are more meager, as considerable bias persists when $\alpha$ is near $-\frac{1}{2}$.

\subsection{Option pricing under rough volatility}\label{sec:roughvol}
As another experiment, we study Monte Carlo option pricing in the \emph{rough Bergomi} (rBergomi) model of Bayer et al.\ \cite{BFG2015}. In the rBergomi model, the logarithmic spot variance of the price of the underlying is modelled by a rough Gaussian process, which is a special case of \eqref{eq:Y}. By virtue of the rough volatility process, the model fits well to observed implied volatility smiles \cite[pp.\ 893--895]{BFG2015}.

More precisely, the price of the underlying in the rBergomi model with time horizon $T>0$ is defined, under an equivalent martingale measure identified with $\mathbb{P}$, as
\begin{equation*}
S(t)  := S(0) \exp \bigg( \int_0^t \sqrt{v(s)} dZ(s) - \frac{1}{2} \int_0^t v(s) ds\bigg), \quad t \in [0,T],
\end{equation*}
using the spot variance process
\begin{equation*}
v(t)  := \xi_0(t) \exp \bigg( \eta \underbrace{\sqrt{2\alpha + 1} \int_0^t (t-s)^\alpha d W(s)}_{=: Y(t)} - \frac{\eta^2}{2} t^{2\alpha + 1}\bigg), \quad t \in [0,T].
\end{equation*}
Above, $S(0)>0$, $\eta >0$ and $\alpha \in (-\frac{1}{2},0)$ are deterministic parameters, and $Z$ is a standard Brownian motion given by
\begin{equation}\label{eq:factor}
Z(t) := \rho W(t) + \sqrt{1-\rho^2} W_\perp(t), \quad t \in [0,T],
\end{equation}
where $\rho \in (-1,1)$ is the correlation parameter and $\{W_\perp(t)\}_{t \in [0,T]}$ is a standard Brownian motion independent of $W$. The process $\{\xi_0(t)\}_{t \in [0,T]}$ is the so-called \emph{forward variance curve} \cite[p.\ 891]{BFG2015}, which we assume here to be flat, $\xi_0(t) = \xi >0$ for all $t \in [0,T]$.

We aim to compute using Monte Carlo simulation the price of a European call option struck at $K>0$ with maturity $T$, which is given by
\begin{equation}\label{eq:option}
\mathrm{C}(S(0),K,T) := \E[(S(T) - K)^+].
\end{equation}
The approach suggested by Bayer et al.\ \cite{BFG2015} involves sampling the Gaussian processes $Z$ and $Y$ on a discrete time grid using exact simulation and then approximating $S$ and $v$ using Euler discretization. We modify this approach by using the hybrid scheme to simulate $Y$, instead of the computationally more costly exact simulation. As the hybrid scheme involves simulating increments of the Brownian motion $W$ driving $Y$, we can conveniently simulate the increments of $Z$, needed for the Euler discretization of $S$, using the representation \eqref{eq:factor}.

\begin{table}[!t]
\vspace{1em} 
\centering 
\begin{tabular}{ccccc}
\toprule
$S(0)$ & $\xi$     & $\eta$ & $\alpha$ & $\rho$ \\
\midrule
$1$    & $0.235^2$ & $1.9$  & $-0.43$  & $-0.9$ \\
\bottomrule
\end{tabular}
\caption{\it Parameter values used in the rBergomi model.}
\label{tab:pars}
\end{table}

We map the option price $\mathrm{C}(S(0),K,T)$ to the corresponding Black--Scholes implied volatility $\mathrm{IV}(S(0),K,T)$, see, e.g., \cite{gatheral2006}. Reparameterizing the implied volatility using the log-strike $k := \log (K/S(0))$ allows us to drop the dependence on the initial price, so we will abuse notation slightly and write $\mathrm{IV}(k,T)$ for the corresponding implied volatility. Figure \ref{fig:IVsmile} displays implied volatility smiles obtained from the rBergomi model using the hybrid and Riemann-sum schemes to simulate $Y$, as discussed above, and compares these to the smiles obtained using an exact simulation of $Y$ via Cholesky factorization. The parameter values are given in Table \ref{tab:pars}. They have been adopted from Bayer et al.\ \cite{BFG2015}, who demonstrate that they result in realistic volatility smiles.
We consider two different maturities: ``short'', $T = 0.041$, and ``long'', $T = 1$. 

\begin{figure}[!t] 
\centering
\begin{tabular}{rl} 
\includegraphics[scale=0.923]{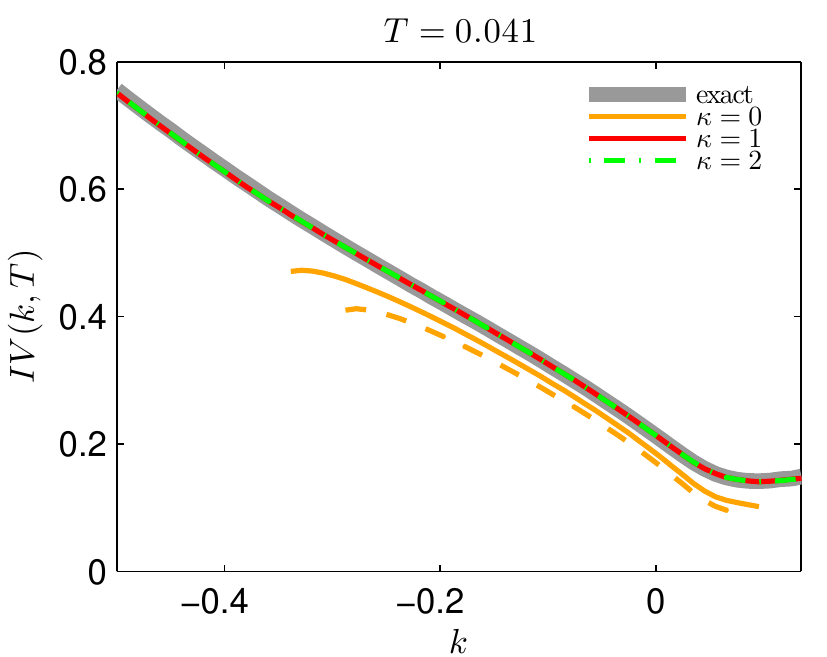} & \includegraphics[scale=0.923]{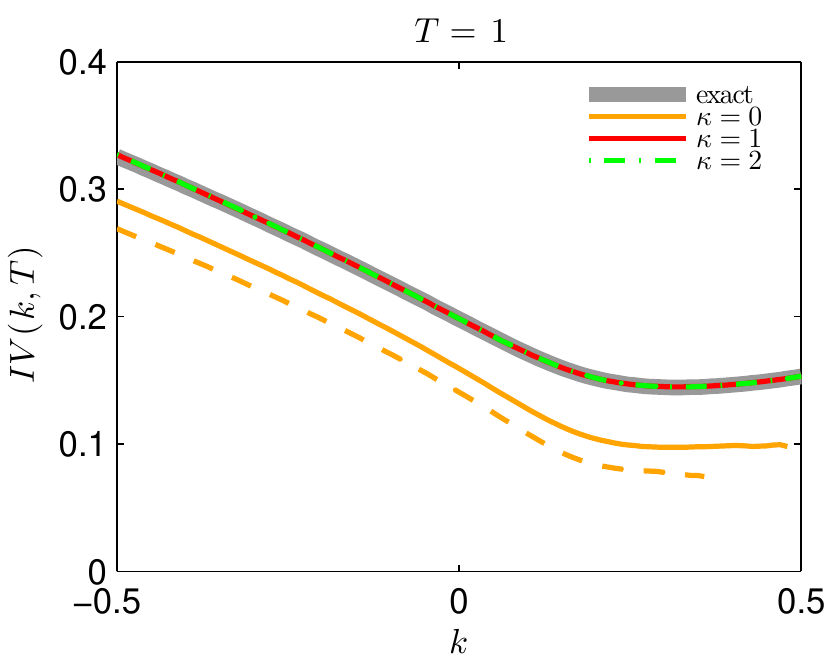}
\end{tabular}
\caption{\it Implied volatility smiles corresponding to the option price \eqref{eq:option}, computed using Monte Carlo simulation ($500$ time steps, $1\,000\,000$ replications), with two maturities: $T = 0.041$ (left) and  $T = 1$ (right).  The spot variance process $v$ was simulated using an exact method, the hybrid scheme ($\kappa=1,2$ and $\mathbf{b} = \mathbf{b}^*$) and Riemann-sum scheme ($\kappa=0$ and $\mathbf{b} = \mathbf{b}^*$ (solid lines), $\mathbf{b}=\mathbf{b}_{\mathrm{FWD}}$ (dashed lines)). The parameter values used in the rBergomi model are given in Table \ref{tab:pars}.}
\label{fig:IVsmile}
\end{figure}

We observe that the Riemann-sum scheme ($\kappa=0$, $\mathbf{b} \in \{  \mathbf{b}_{\mathrm{FWD}},\mathbf{b}^*\}$) is able capture the shape of the implied volatility smile, but not its level. Alas, the method even breaks down with more extreme log-strikes (the prices are so low that the root-finding algorithm used to compute the implied volatility would return zero). 
 In contrast, the hybrid scheme with $\kappa = 1,2$ and $\mathbf{b}=\mathbf{b}^*$ yields implied volatility smiles that are indistinguishable from the benchmark smiles obtained using exact simulation. Further, there is no discernible difference between the smiles obtained using $\kappa=1$ and $\kappa=2$.
 As in the previous section, we observe that the hybrid scheme is indeed capable of producing very accurate trajectories of $\mathcal{TBSS}$ processes, in particular in the case $\alpha \in (-\frac{1}{2},0)$, even when $\kappa = 1$.

\section{Proofs}\label{sec:proofs}

Throughout the proofs below, we rely on two useful inequalities. The first one is the Potter bound for slow variation at $0$, which follows immediately from the corresponding result for slow variation at $\infty$ \cite[Theorem 1.5.6]{BGT}. Namely, if $L : (0,1] \rightarrow (0,\infty)$ is slowly varying at $0$ and bounded away from $0$ and $\infty$ on any interval $(u,1]$, $u \in (0,1)$, then for any $\delta>0$ there exists a constant $C_\delta >0$ such that
\begin{equation}\label{eq:potterbounds}
\frac{L(x)}{L(y)} \leq C_\delta \max\bigg\{\Big( \frac{x}{y}\Big)^\delta, \Big(\frac{x}{y}\Big)^{-\delta}\bigg\}, \quad x,y \in (0,1].
\end{equation}
The second one is the elementary inequality
\begin{equation}\label{eq:powerbounds}
|x^\alpha - y^\alpha| \leq |\alpha|(\min \{x,y\})^{\alpha-1}|x-y|, \quad x,y \in (0,\infty), \quad \alpha \in (-\infty,1),
\end{equation}
which can be easily shown using the mean value theorem. Additionally, we use the following variant of Karamata's theorem for regular variation at $0$. Its proof is similar to the one of the usual Karamata's theorem  for regular variation at $\infty$ \cite[Proposition 1.5.10]{BGT}.

\begin{lemma}[Karamata's theorem]\label{lem:karamatazero}
If $\alpha \in (-1,\infty)$ and $L : (0,1] \rightarrow [0,\infty)$ is slowly varying at $0$, then
\begin{equation*}
\int_0^y x^\alpha L(x) dx \sim \frac{1}{\alpha+1} y^{\alpha + 1} L(y), \quad y \rightarrow 0.
\end{equation*}
\end{lemma}

\subsection{Proof of Proposition \ref{prop:localbehavior}}\label{ssec:proofs1} 

\begin{proof}[Proof of Proposition \ref{prop:localbehavior}]
\eqref{item:vario} By the covariance stationarity of the volatility process $\sigma$, we may express the variogram $V(h)$ for any $h \geq 0$ as
\begin{equation}\label{eq:variogramformula}
\begin{split}
V(h) = \mathbb{E}[|X(h) - X(0)|^2] & = \int_{-\infty}^h \big(g(h-u) - g(-u)\mathbf{1}_{(-\infty,0)}(u)\big)^2 \mathbb{E}[\sigma(u)^2] du \\
& = \mathbb{E}[\sigma(0)^2] \bigg(\int_{0}^h g(x)^2 d x  + \int_0^\infty (g(x+h)-g(x))^2 dx\bigg). 
\end{split}
\end{equation}
Invoking \eqref{ass:zero} and Lemma \ref{lem:karamatazero}, we find that
\begin{equation}\label{eq:gnearzero}
\int_{0}^h g(x)^2 d x \sim \frac{1}{2\alpha + 1} h^{2\alpha + 1} L_g(h)^2, \quad h \rightarrow 0.
\end{equation}
We may clearly assume that $h<1$, which allows us to work with the decomposition
\begin{equation*}
\int_0^\infty (g(x+h)-g(x))^2 dx = A_h + A'_h,
\end{equation*}
where
\begin{equation*}
A_h :=\int_{0}^{1-h} (g(x+h)-g(x))^2 dx, \quad A'_h := \int_{1-h}^{\infty} (g(x+h)-g(x))^2 dx.
\end{equation*}

According to \eqref{ass:deriv}, there exists $M>1$ such that $x \mapsto |g'(x)|$ is non-increasing on $[M,\infty)$. Thus, using the mean value theorem, we deduce that
\begin{equation*}
|g(x+h) - g(x)| = |g'(\xi)|h \leq \begin{cases}
\sup_{y \in (1-h,M]}|g'(y)| h, & x \in (1-h,M),\\
|g'(x)|h, & x \in [M,\infty).
\end{cases}
\end{equation*}
where $\xi = \xi(x,h) \in [x,x+h]$. It follows then that
\begin{equation*}
\limsup_{h \rightarrow 0} \frac{A'_h}{h^2} \leq (M-1) \sup_{y \in [1,M]}g'(y)^2 + \int_1^\infty g'(x)^2 dx < \infty,
\end{equation*}
which in turn implies that
\begin{equation}\label{eq:gnegligible}
A'_h = \mathcal{O}(h^2), \quad h \rightarrow 0.
\end{equation}

Making a substitution $y=\frac{x}{h}$, we obtain
\begin{equation*}
\begin{split}
A_h = \int_{0}^{1-h} (g(x+h)-g(x))^2 dx & = h \int_{0}^{1/h-1} \big(g(h(y+1))-g(hy)\big)^2 dy \\
& = h^{2\alpha + 1} L_g(h)^2\int_{0}^{\infty} G_h(y) dy, 
\end{split}
\end{equation*}
where
\begin{equation*}
G_h(y) := \bigg((y+1)^\alpha \frac{L_g(h(y+1))}{L_g(h)}-y^\alpha\frac{L_g(hy)}{L_g(h)}\bigg)^2\mathbf{1}_{(0,1/h-1)}(y), \quad y \in (0,\infty).
\end{equation*}
By the definition of slow variation at $0$,
\begin{equation*}
\lim_{h \rightarrow 0} G_h(y) = \big((y+1)^\alpha-y^\alpha\big)^2, \quad \quad y \in (0,\infty).
\end{equation*}
We shall show below that the functions $G_h$, $h \in (0,1)$, have an integrable dominant. Thus, by the dominated convergence theorem,
\begin{equation}\label{eq:dctappl}
A_h \sim h^{2\alpha + 1 } L_g(h)^2 \int_0^\infty \big((y+1)^\alpha-y^\alpha\big)^2 dy, \quad h \rightarrow 0. 
\end{equation}
Since $\alpha < \frac{1}{2}$, we have $\lim_{h \rightarrow 0} \frac{A'_h}{h^{2\alpha + 1 } L_g(h)^2}=0$ by \eqref{eq:potterzero} and \eqref{eq:gnegligible}, so we get from \eqref{eq:gnearzero} and \eqref{eq:dctappl}
\begin{multline*}
\int_{0}^h g(x)^2 d x  + \int_0^\infty (g(x+h)-g(x))^2 dx \\ \sim \bigg(\frac{1}{2\alpha + 1} + \int_0^\infty \big( (y+1)^\alpha - y^\alpha\big)^2 dy\bigg) h^{2\alpha+1} L_g(h)^2, \quad h \rightarrow 0,
\end{multline*}
which, together with \eqref{eq:variogramformula}, implies the assertion.

It remains to justify the use of the dominated convergence theorem to deduce \eqref{eq:dctappl}.
For any $y \in (0,1]$, we have by the Potter bound \eqref{eq:potterbounds} and the elementary inequality $(u+v)^2 \leq 2u^2+2v^2$,
\begin{equation*}
\begin{split}
G_h(y) & \leq 2(y+1)^{2\alpha} \bigg(\frac{L_g(h(y+1))}{L_g(h)}\bigg)^2 + 2y^{2\alpha} \bigg(\frac{L_g(hy)}{L_g(h)}\bigg)^2 \\
& \leq 2C^2_{\delta_1}\big((y+1)^{2(\alpha+\delta_1)} + y^{2(\alpha-\delta_1)} \big),
\end{split}
\end{equation*}
where we choose $\delta_1 \in (0,\alpha+\frac{1}{2})$ to ensure that $2(\alpha-\delta_1) > -1$.
Consider then $y \in [1,\infty)$. By adding and substracting the term $(y+1)^\alpha \frac{L_g(hy)}{L_g(h)}$ and using again the inequality $(u+v)^2 \leq 2u^2+2v^2$, we get   
\begin{equation*}
\begin{split}
G_h(y) 
 & = \bigg((y+1)^\alpha \frac{L_g(h(y+1))}{L_g(h)} -(y+1)^\alpha \frac{L_g(hy)}{L_g(h)} \\
 & \quad +(y+1)^\alpha \frac{L_g(hy)}{L_g(h)} - y^\alpha \frac{L_g(hy)}{L_g(h)}\bigg)^2\mathbf{1}_{(0,1/h-1)}(y)\\
& \leq 2 (y+1)^{2\alpha}\bigg(\frac{L_g(h(y+1))-L_g(hy)}{L_g(h)}\mathbf{1}_{(0,1/h-1)}(y)\bigg)^2\\
& \qquad + 2\big((y+1)^\alpha-y^\alpha\big)^2 \bigg(\frac{L_g(hy)}{L_g(h)}\mathbf{1}_{(0,1/h-1)}(y)\bigg)^2 .
\end{split}
\end{equation*}
We recall that $\underline{L}_g := \inf_{x \in (0,1]} L_g(x) >0$ by \eqref{ass:zero}, so
\begin{equation*}
\bigg| \frac{L_g(h(y+1))-L_g(hy)}{L_g(h)}\mathbf{1}_{(0,1/h-1)}(y)\bigg| \leq \frac{1}{\underline{L}_g} |L_g(h(y+1))-L_g(hy)|\mathbf{1}_{(0,1/h-1)}(y).
\end{equation*}
Using the mean value theorem and the bound for the derivative of $L_g$ from \eqref{ass:zero}, we observe that
\begin{equation*}
|L_g(h(y+1))-L_g(hy)| = |L'_g(\xi)| |h(y+1)-hy| \leq h C \bigg(1+\frac{1}{\xi}\bigg) \leq C \bigg(h+\frac{1}{y}\bigg),
\end{equation*}
where $\xi = \xi(y,h) \in [hy,h(y+1)]$. Noting that the constraint $y < \frac{1}{h}-1$ is equivalent to $h < \frac{1}{y+1}$, we obtain further
\begin{equation*}
\bigg| \frac{L_g(h(y+1))-L_g(hy)}{L_g(h)}\mathbf{1}_{(0,1/h-1)}(y)\bigg| \leq \frac{C}{\underline{L}_g} \bigg(h+\frac{1}{y}\bigg)\mathbf{1}_{(0,1/h-1)}(y) \leq \frac{C}{\underline{L}_g} \bigg(\frac{1}{y+1}+ \frac{1}{y}\bigg) \leq \frac{C}{\underline{L}_g}\frac{3}{ y+1},
\end{equation*}
as $y\geq 1$, which we then use to deduce that
\begin{equation*}
 2 (y+1)^{2\alpha}\bigg(\frac{L_g(h(y+1))-L_g(hy)}{L_g(h)}\mathbf{1}_{(0,1/h-1)}(y)\bigg)^2 \leq \frac{18 C^2}{\underline{L}^2_g}(y+1)^{2(\alpha-1)}.
\end{equation*}
Additionally, we observe that, by \eqref{eq:potterbounds} and \eqref{eq:powerbounds},
\begin{equation*}
2\big((y+1)^\alpha-y^\alpha\big)^2 \bigg(\frac{L_g(hy)}{L_g(h)}\mathbf{1}_{(0,1/h-1)}(y)\bigg)^2 \leq 2 C^2_{\delta_2} \alpha^2 y^{2(\alpha-1 +\delta_2)},
\end{equation*}
where we choose $\delta_2 \in (0,\frac{1}{2}-\alpha)$, ensuring that $2(\alpha-1+\delta_2) < -1$. We may finally define a function
\begin{equation*}
G(y) := \begin{cases}
2C^2_{\delta_1}\big((y+1)^{2(\alpha+\delta_1)} + y^{2(\alpha-\delta_1)} \big), & y \in (0,1],\\
\frac{18 C^2}{\underline{L}^2_g}(y+1)^{2(\alpha-1)}+ 2 C^2_{\delta_2} \alpha^2  y^{2(\alpha-1 +\delta_2)}, & y \in (1,\infty),
\end{cases}
\end{equation*}
which satisfies $0 \leq G_h(y) \leq G(y)$ for any $y \in (0,\infty)$ and $h \in (0,1)$, and is integrable on $(0,\infty)$ with the aforementioned choices of $\delta_1$ and $\delta_2$.

\eqref{item:holder} 
To show existence of the modification, we need a localization procedure that involves an ancillary process
\begin{equation*}
F(t) := \int_1^\infty g'(s)^2 \sigma(t-s)^2 ds, \quad t \in \R.
\end{equation*}
We check first that $F$ is locally bounded under \eqref{ass:zero} and \eqref{ass:deriv}, which is essential for localization. To this end, let $T \in (0,\infty)$, and write for any $t \in [-T,T]$,
\begin{equation*}
F(t) = F^\flat(t) + F^\sharp(t),
\end{equation*}
where
\begin{equation*}
F^\flat(t) := \int_1^{M+2T} g'(s)^2 \sigma(t-s)^2 ds, \quad F^\sharp(t) := \int_{M+2T}^\infty g'(s)^2 \sigma(t-s)^2 ds,
\end{equation*}
and $M>1$ is such that $x \mapsto |g'(x)|$ is non-increasing on $[M,\infty)$, as in the proof of \eqref{item:vario}.

Since $g'$ is continuous on $(0,\infty)$ and $\sigma$ locally bounded, we have for any $t \in [-T,T]$,
\begin{equation*}
0 \leq F^\flat(t) \leq (M+2T-1) \sup_{y \in [1,M+2T]} g'(y)^2 \sup_{u \in [-M-3T,T-1]}\sigma(u)^2 <\infty.
\end{equation*}
Further, when $t \in [-T,T]$,
\begin{equation*}
F^\sharp(t) = \int_{-\infty}^{t-(M+2T)} g'(t-u)^2 \sigma(u)^2 du, 
\end{equation*}
where $g'(t-u)^2 \leq g'(-T-u)^2$ since the arguments satisfy
\begin{equation*}
t-u \geq -T-u \geq -T-\big(t-(M+2T)\big) \geq M.
\end{equation*}
Thus,
\begin{equation*}
0 \leq F^\sharp(t) \leq \int_{-\infty}^{-(M+T)} g'(-T-u)^2 \sigma(u)^2 du \leq \int_{1}^{\infty} g'(s)^2 \sigma(-T-s)^2 ds <\infty
\end{equation*}
for any $t \in [-T,T]$ almost surely, as we have \begin{equation*}
\mathbb{E}\bigg[\int_1^\infty g'(s)^2 \sigma(-T-s)^2 ds\bigg] = \int_1^\infty g'(s)^2 \mathbb{E}[\sigma(-T-s)^2] ds = \E[\sigma(0)^2] \int_1^\infty g'(s)^2 ds < \infty,
\end{equation*}
where we change the order of expectation and integration relying on Tonelli's theorem and where the final equality follows from the covariance stationarity of $\sigma$. So we can conclude that $F$ is indeed locally bounded.

Let now $m \in \N$ and, for localization, define a sequence of stopping times
\begin{equation*}
\tau_{m,n} := \inf\{ t \in [-m,\infty) : F(t) > n \textrm{ or } |\sigma(t)| > n \}, \quad n \in \N,
\end{equation*}
that satisfies $\tau_{m,n} \uparrow \infty$ almost surely as $n \rightarrow \infty$ since both $F$ and $\sigma$ are locally bounded. (We follow the usual convention that $\inf \varnothing = \infty$.) Consider now the modified $\BSS$ process
\begin{equation*}
X^\dag_{m,n}(t) := \int_{-\infty}^t g(t-s) \sigma(\min \{s,\tau_{m,n}\}) d W(s), \quad t \in [-m,\infty),
\end{equation*}
that coincides with $X$ on the stochastic interval $\llbracket -m,\tau_{m,n} \rrbracket$. The process $X^\dag_{m,n}$ satisfies the assumptions of \cite[Lemma 1]{ole_jose_mark11}, so for any $p>0$ there exists a constant $\hat{C}_p>0$ such that 
\begin{equation}\label{eq:varioloc}
\mathbb{E}[|X^\dag_{m,n}(s)-X^\dag_{m,n}(t)|^p] \leq \hat{C}_p V(|s-t|)^{p/2}, \quad s,t \in [-m,\infty).
\end{equation}

Using the upper bound in \eqref{eq:potterzero}, we can deduce from \eqref{item:vario} that for any $\delta > 0$ there are constants $\tilde{C}_\delta>0$ and $\underline{h}_\delta > 0$ such that
\begin{equation}\label{eq:varioupper}
V(h) \leq \tilde{C}_\delta h^{2\alpha + 1 - \delta}, \quad h \in (0,\underline{h}_\delta).
\end{equation}
Applying \eqref{eq:varioupper} to \eqref{eq:varioloc}, we get
\begin{equation*}
\mathbb{E}[|X^\dag_{m,n}(s)-X^\dag_{m,n}(t)|^p] \leq \hat{C}_p\tilde{C}^{p/2}_\delta |s-t|^{1+p(\alpha + \frac{1}{2}-\frac{\delta}{2}-\frac{1}{p})}, \quad s,t \in [-m,\infty), \quad |s-t| < \underline{h}_\delta.
\end{equation*}
We may note that $p(\alpha + \frac{1}{2}-\frac{\delta}{2}-\frac{1}{p})>0$ for small enough $\delta$ and large enough $p$ and, in particular,
\begin{equation*}
\frac{p(\alpha + \frac{1}{2}-\frac{\delta}{2}-\frac{1}{p})}{p} \uparrow \alpha + \frac{1}{2},
\end{equation*}
as $\delta \downarrow 0$ and $p \uparrow \infty$. Thus it follows from the Kolmogorov--Chentsov theorem \cite[Theorem 3.22]{kallenberg} that $X^\dag_{m,n}$ has a modification with locally $\phi$-H\"older continuous trajectories for any $\phi \in (0,\alpha + \frac{1}{2})$. Moreover, a modification of $X$ on $\R$, having locally $\phi$-H\"older continuous trajectories for any $\phi \in (0,\alpha + \frac{1}{2})$, can then by constructed from these modifications of $X^\dag_{m,n}$, $m \in \N$, $n \in \N$, by letting first $n \rightarrow \infty$ and then $m \rightarrow \infty$. 
\end{proof}

\subsection{Proof of Theorem \ref{th:mainTh}}\label{ssec:proofs2}  

As a preparation, we shall first establish an auxiliary result that deals with the asymptotic behavior of certain integrals of regularly varying functions.

\begin{lemma}\label{lem:lemma2}
Suppose that $L : (0,1] \rightarrow [0,\infty)$ is bounded away from $0$ and $\infty$ on any set of the form $(u,1]$, $u\in(0,1)$, and slowly varying at $0$. Moreover, let $\alpha \in (-\frac{1}{2},\infty) $ and $k \geq 1$. If $b \in [k-1,k] \setminus \{0\}$, then
\begin{enumerate}[label=(\roman*),ref=\roman*,leftmargin=*] 
\item\label{item:int1} ${\displaystyle \lim_{n \rightarrow \infty}\int_{k-1}^k \bigg( x^{\alpha}\frac{L(x/n)}{L(1/n)} - b^{\alpha}\frac{L(b/n)}{L(1/n)} \bigg)^2 dx = \int_{k-1}^k (x^{\alpha} - b^{\alpha})^2 dx<\infty}$,
\item\label{item:int2}${\displaystyle \lim_{n \rightarrow \infty}\int_{k-1}^k x^{2\alpha}\bigg( \frac{L(x/n)}{L(1/n)} - \frac{L(b/n)}{L(1/n)} \bigg)^2 dx = 0}$.
\end{enumerate}
\end{lemma}
\begin{proof}
We only prove \eqref{item:int1} as \eqref{item:int2} can be shown similarly. By the definition of slow variation at $0$, the function
\begin{equation*}
f_n(x) := \bigg( x^{\alpha}\frac{L(x/n)}{L(1/n)} - b^{\alpha} \frac{L(b/n)}{L(1/n)}\bigg)^2, \quad x \in [k-1,k] \setminus \{0 \},
\end{equation*}
satisfies $\lim_{n \rightarrow\infty} f_n(x) = (x^{\alpha} - b^{\alpha})^2$ for any $x \in [k-1,k] \setminus \{0 \}$. In view of the dominated convergence theorem, it suffices to find an integrable dominant for the functions $f_n$, $n \in \N$. The construction of the dominant is quite similar to the one seen in the proof of Proposition \ref{prop:localbehavior}, but we provide the details for the convenience of the reader. 

Using the Potter bound \eqref{eq:potterbounds} and the inequality $(u+v)^2 \leq 2u^2+2v^2$, we find that for any $x \in [k-1,k] \setminus \{0 \}$,
\begin{equation*}
\begin{split}
0 \leq f_n(x) & \leq 2 x^{2\alpha}\bigg(\frac{L(x/n)}{L(1/n)}\bigg)^2 + 2b^{2\alpha}\bigg(\frac{L(b/n)}{L(1/n)}\bigg)^2  \\
& \leq 2C^2_\delta \Big( x^{2\alpha} \max\big\{x^\delta, x^{-\delta}\big\}^2 + b^{2\alpha} \max\big\{b^\delta, b^{-\delta}\big\}^2\Big) =: f(x),
\end{split}
\end{equation*}
where we choose $\delta \in (0,\alpha+\frac{1}{2})$. When $k \geq 2$, we have $x\geq 1$ and $b \geq 1$, so
\begin{equation*}
f(x) = 2C^2_\delta \big( x^{2(\alpha+\delta)}  + b^{2(\alpha+\delta)}\big)
\end{equation*}
is a bounded function of $x$ on $[k-1,k]$. When $k=1$, we have $x\leq 1$ and $b \leq 1$, implying that
\begin{equation*}
f(x) = 2C^2_\delta \big( x^{2(\alpha-\delta)}  + b^{2(\alpha-\delta)}\big),
\end{equation*}
where $2(\alpha-\delta) > - 1$ with our choice of $\delta$, so $f$ is an integrable function on $(0,1]$.
\end{proof}

\begin{proof}[Proof of Theorem \ref{th:mainTh}] Let $t \in \R$ be fixed. It will be convenient to write $X_n(t)$ as
\begin{equation*}
\begin{split}
X_n(t) & = \sum_{k=1}^\kappa \int_{t-\frac{k}{n}}^{t - \frac{k-1}{n}}  (t-s)^\alpha L_g\bigg(\frac{k}{n}\bigg) \sigma\bigg(t-\frac{k}{n}\bigg) d W(s)\\ & \quad + \sum_{k=\kappa+1}^{N_n} \int_{t-\frac{k}{n}}^{t - \frac{k-1}{n}} g\bigg(\frac{b_k}{n} \bigg) \sigma\bigg(t-\frac{k}{n}\bigg) d W(s).
\end{split}
\end{equation*}
Moreover, we introduce an ancillary approximation of $X(t)$, namely,
\begin{align*}
X'_n(t) = \sum_{k=1}^{N_n} \int_{t-\frac{k}{n}}^{t - \frac{k-1}{n}} g(t - s) \sigma\bigg(t-\frac{k}{n}\bigg) d W(s) + \int_{-\infty}^{t - \frac{N_n}{n}} g(t-s) \sigma(s) d W(s).
\end{align*}
By Minkowski's inequality, we have
\begin{align*}
\mathbb{E}\big[| X_n(t) - X(t) |^2\big]^{\frac{1}{2}} & \geq \mathbb{E}\big[| X_n(t) - X'_n (t)|^2\big]^{\frac{1}{2}} - \mathbb{E}\big[| X'_n(t) - X (t) |^2\big]^{\frac{1}{2}}, \\
\mathbb{E}\big[| X_n(t) - X(t) |^2\big]^{\frac{1}{2}} & \leq \mathbb{E}\big[| X_n(t) - X'_n (t)|^2\big]^{\frac{1}{2}} + \mathbb{E}\big[| X'_n(t) - X (t) |^2\big]^{\frac{1}{2}},
\end{align*}
which together, after taking squares, imply that
\begin{align}\label{eq:minkowski}
E_n\Bigg(1-2\sqrt{\frac{E'_n}{E_n}} + \frac{E'_n}{E_n}\Bigg)  \leq \mathbb{E}\big[| X_n(t) - X(t) |^2\big] \leq E_n \Bigg(1+2\sqrt{\frac{E'_n}{E_n}} + \frac{E'_n}{E_n}\Bigg),
\end{align}
where
\begin{align*}
E_n  := \mathbb{E}\big[| X_n(t) - X'_n (t) |^2\big], \quad E'_n  := \mathbb{E}\big[| X(t) - X'_n (t) |^2\big].
\end{align*}
Using the It\=o isometry, and recalling that $\sigma$ is covariance stationary, we obtain
\begin{equation*}
\begin{split}
E'_n & = \sum_{k=1}^{N_n} \int_{t-\frac{k}{n}}^{t - \frac{k-1}{n}} g(t - s)^2 \mathbb{E}\bigg[\bigg(\sigma\bigg(t-\frac{k}{n}\bigg)-\sigma(s)\bigg)^2\bigg]  d s \\
& \leq \sup_{u \in (0,\frac{1}{n}]} \mathbb{E}\big[|\sigma(u)-\sigma(0)|^2\big] \int_0^\infty g(s)^2 ds
\end{split}
\end{equation*}
and
\begin{equation*}
\begin{split}
E_n & = \sum_{k=1}^\kappa \int_{t-\frac{k}{n}}^{t-\frac{k-1}{n}} \bigg((t-s)^\alpha L_g\bigg(\frac{k}{n}\bigg) - g(t-s)\bigg)^2\mathbb{E}\bigg[\sigma\bigg(t-\frac{k}{n} \bigg)^2\bigg] ds \\
& \quad + \sum_{k = \kappa + 1}^{n} \int_{t-\frac{k}{n}}^{t-\frac{k-1}{n}}\bigg(g\bigg(\frac{b_k}{n}\bigg) - g(t-s)\bigg)^2\mathbb{E}\bigg[\sigma\bigg(t-\frac{k}{n} \bigg)^2\bigg] ds \\
& \quad + \sum_{k = n + 1}^{N_n} \int_{t-\frac{k}{n}}^{t-\frac{k-1}{n}}\bigg(g\bigg(\frac{b_k}{n}\bigg) - g(t-s)\bigg)^2\mathbb{E}\bigg[\sigma\bigg(t-\frac{k}{n} \bigg)^2\bigg] ds \\
& \quad + \int_{-\infty}^{t - \frac{N_n}{n}} g(t-s)^2 \mathbb{E}[\sigma(s)^2] d s \\
& = \E[\sigma(0)^2](D_n + D'_n +D''_n + D'''_n),
\end{split}
\end{equation*}
where
\begin{align*}
D_n &:= \sum_{k=1}^\kappa \int_{\frac{k-1}{n}}^{\frac{k}{n}} \bigg( s^\alpha L_g \bigg(\frac{k}{n}\bigg)-g(s)\bigg)^2 ds,&
D'_n &:= \sum_{k=\kappa+1}^n \int_{\frac{k-1}{n}}^{\frac{k}{n}} \bigg(g\bigg(\frac{b_k}{n}\bigg) - g(s)\bigg)^2ds, \\
D_n'' &:=  \sum_{k=n+1}^{N_n}\int_{\frac{k-1}{n}}^{\frac{k}{n}} \bigg(g\bigg(\frac{b_k}{n}\bigg) - g(s)\bigg)^2ds, &
D_n''' &:= \int_{\frac{N_n}{n}}^{\infty} g(s)^2ds.
\end{align*}
(We may assume without loss of generality that $N_n > n >\kappa$, as this will be the case for large enough $n$.)
In what follows, we study the asymptotic behavior of the terms $D_n$, $D'_n$, $D''_n$ and $D'''_n$  separately, showing that $D_n$, $D''_n$ and $D'''_n$ are negligible in comparison with $D'_n$, and that $D'_n$ gives rise to the convergence rate given in Theorem \ref{th:mainTh}.

Let us analyze the terms $D'''_n$, $D''_n$ and $D_n$ first. By \eqref{ass:infb} and \eqref{ass:truncation}, we have
\begin{equation}\label{eq:cprimeprimea}
D_n''' = \mathcal{O}\Bigg(\bigg(\frac{N_n}{n}\bigg)^{2\beta +1}\Bigg) = \mathcal{O} \big(n^{\gamma(2 \beta+1)}\big), \quad n \rightarrow \infty.
\end{equation}
Regarding the term $D_n''$, recall that by \eqref{ass:deriv} there is $M>1$ such that $x \mapsto |g'(x)|$ is non-increasing on $[M,\infty)$. So, we have by the mean value theorem,
\begin{equation*}
\bigg| g\bigg(\frac{b_k}{n}\bigg) - g(s)\bigg| = |g'(\xi)|\bigg|\frac{b_k}{n}-s \bigg| \leq \begin{cases}
\frac{1}{n} \sup_{y \in [1,M]} |g'(y)|, & \frac{k-1}{n} < M,\\
\frac{1}{n}|g'\big( \frac{k-1}{n}\big)|, & \frac{k-1}{n} \geq M,
\end{cases}
\end{equation*}
where $\xi = \xi(\frac{b_k}{n},s) \in [\frac{k-1}{n},\frac{k}{n}]$.
Thus,
\begin{equation*}
\begin{split}
\limsup_{n \rightarrow \infty} n^2 D_n'' & \leq (M-1) \sup_{y \in [1,M]} g'(y)^2 +  \int_1^{\infty}g'(s)^2ds < \infty,
\end{split}
\end{equation*}
which implies
\begin{equation}\label{eq:cprime}
D_n'' = \mathcal{O}(n^{-2}), \quad n \rightarrow \infty.
\end{equation}
To analyze the behavior of $D_n$, we substitute $y=ns$ and invoke \eqref{ass:zero}, yielding
\begin{equation*}
\begin{split}
D_n & = \sum_{k=1}^\kappa \int_{k-1}^k \Bigg( \bigg(\frac{y}{n}\bigg)^\alpha L_g\bigg(\frac{k}{n}\bigg)- g\bigg(\frac{y}{n}\bigg)\Bigg)^2 \frac{dy}{n} \\
& = n^{-(2\alpha+1)}L_g(1/n)^2 \sum_{k=1}^\kappa \int_{k-1}^k y^{2\alpha} \bigg(\frac{L_g(k/n)}{L_g(1/n)}-\frac{L_g(y/n)}{L_g(1/n)} \bigg)^2 dy,
\end{split}
\end{equation*}
where, by Lemma \ref{lem:lemma2}\eqref{item:int2}, we have
\begin{equation*}
\lim_{n \rightarrow \infty}\int_{k-1}^k y^{2\alpha} \bigg(\frac{L_g(k/n)}{L_g(1/n)}-\frac{L_g(y/n)}{L_g(1/n)} \bigg)^2 dy = 0
\end{equation*}
for any $k = 1,\ldots,\kappa$.
Thus, we find that
\begin{equation}\label{eq:exact-negligible}
\lim_{n \rightarrow \infty} \frac{D_n}{n^{-(2\alpha+1)}L_g(1/n)^2} = 0.
\end{equation}

The asymptotic behavior of the term $D'_n$ is more delicate to analyze. By \eqref{ass:zero}, and substituting $y = ns$, we can write
\begin{align*}
D'_n &= \sum_{k=\kappa+1}^n \int_{k-1}^{k} \bigg(g\bigg(\frac{b_k}{n}\bigg) - g\bigg(\frac{y}{n}\bigg)\bigg)^2\frac{dy}{n} \\
	&= n^{-(2\alpha+1)} \sum_{k=\kappa+1}^n \int_{k-1}^{k} \bigg(b_k^{\alpha} L_g\bigg(\frac{b_k}{n}\bigg) - y^{\alpha}L_g\bigg(\frac{y}{n}\bigg)\bigg)^2 dy \\
	&= n^{-(2\alpha+1)} L_g(1/n)^2 \sum_{k=\kappa+1}^n A_{n,k},
\end{align*}
where
\begin{equation*}
A_{n,k} :=  \int_{k-1}^{k} \bigg(y^{\alpha} \frac{L_g(y/n)}{L_g(1/n)} - b_k^{\alpha}\frac{L_g(b_k/n)}{L_g(1/n)} \bigg)^2 dy.
\end{equation*}
Let us study the asymptotic behavior of the sum $\sum_{k=\kappa+1}^n A_{n,k}$ as $n \rightarrow \infty$. By Lemma \ref{lem:lemma2}, we have for any $k \in \N$,
\begin{equation*}
\lim_{n \rightarrow \infty} A_{n,k} = \int_{k-1}^k (y^{\alpha} - b_k^{\alpha})^2 dy<\infty.
\end{equation*}
To be able to then deduce, using the dominated convergence theorem, that
\begin{equation}\label{eq:Aconv}
\lim_{n \rightarrow \infty} \sum_{k=\kappa+1}^n A_{n,k} = \sum_{k=1}^\infty \int_{k-1}^k (y^{\alpha} - b_k^{\alpha})^2 dy = J(\alpha,\kappa,\mathbf{b}) < \infty,
\end{equation}
we seek a sequence $\{A_k\}_{k=\kappa+1}^{\infty} \subset [0,\infty)$ such that
\begin{align*}
0 \leq A_{n,k} \leq A_k, \quad k = \kappa+1,\ldots,n, \quad n \in \N.
\end{align*}
and that $\sum_{k=\kappa+1}^\infty A_k < \infty$. Let us assume, without loss of generality, that $\kappa = 0$.
Clearly, we may set $A_1 := \sup_{n \in \N} A_{n,1} < \infty$.
Consider now $k\geq 2$. The construction of $A_k$ in this case parallels some arguments seen in the proof of Proposition \ref{prop:localbehavior}, but we provide the details for the sake of clarity. By adding and substracting $b^{\alpha}_k \frac{L_g(y/n)}{L_g(1/n)}$ and using the inequality $(u+v)^2 \leq 2u^2 + 2v^2$, we get
\begin{equation}\label{eq:Adecomp}
\begin{split}
A_{n,k}   & = \bigg( y^\alpha \frac{L_g(y/n)}{L_g(1/n)} - b_k^\alpha \frac{L_g(y/n)}{L_g(1/n)} +  b_k^\alpha \frac{L_g(y/n)}{L_g(1/n)}  -  b_k^\alpha \frac{L_g(b_k/n)}{L_g(1/n)} \bigg)^2 
\\ &  \leq 2\int_{k-1}^k (y^{\alpha}-b^\alpha_k)^2 \bigg(\frac{L_g(y/n)}{L_g(1/n)}\bigg)^2 dy + 2 b^{2\alpha}_k\int_{k-1}^k\bigg(\frac{L_g(y/n) - L_g(b_k/n)}{L_g(1/n)}\bigg)^2dy =: I_{n,k} + I'_{n,k}.
\end{split}
\end{equation}
Recall that $\underline{L}_g := \inf_{x \in (0,1]} L_g(x)>0$, so by the estimates $b^{2\alpha}_k \leq \max \{k^{2\alpha},(k-1)^{2\alpha} \} \leq 2 (k-1)^{2\alpha}$, valid when $\alpha < \frac{1}{2}$, we obtain\begin{align*}
I'_{n,k}  \leq \frac{4}{\underline{L}_g^2} (k-1)^{2\alpha} \int_{k-1}^k \bigg(L_g(y/n) - L_g(b_k/n)\bigg)^2dy.
\end{align*}
Note that, thanks to \eqref{ass:zero} and the mean value theorem, 
\begin{align*}
|L_g(y/n) - L_g(b_k/n)| = |L_g'(\xi)| \bigg| \frac{y}{n} - \frac{b_k}{n}\bigg| \leq \frac{C(1+\xi^{-1})}{n}\leq \frac{C}{n} + \frac{C}{k-1} \leq \frac{2C}{k-1},
\end{align*}
where $\xi = \xi(y/n,b_k/n) \in [\frac{k-1}{n},\frac{k}{n}]$ and where the final inequality follows since $k-1 < n$.
Thus,
\begin{equation}\label{eq:Ibound1}
I'_{n,k}  \leq \frac{16 C^2}{\underline{L}_g^2} (k-1)^{2(\alpha-1)}.
\end{equation}
Moreover, 
the Potter bound \eqref{eq:potterbounds} and inequality \eqref{eq:powerbounds} imply
\begin{equation}\label{eq:Ibound2}
I_{n,k} \leq 2\alpha^2C^2_\delta \int_{k-1}^k (\min\{y,b_k \})^{2(\alpha-1)} y^{2\delta} dy \leq 2^{1+2\delta} \alpha^2C^2_\delta(k-1)^{2(\alpha-1+\delta)},
\end{equation}
where we choose $\delta \in (0,\frac{1}{2}-\alpha)$.
Applying the bounds \eqref{eq:Ibound1} and \eqref{eq:Ibound2} to \eqref{eq:Adecomp} shows that
\begin{equation*}
A_{n,k} \leq 2^{1+2\delta} \alpha^2C^2_\delta(k-1)^{2(\alpha-1+\delta)} + \frac{16 C^2}{\underline{L}_g^2} (k-1)^{2(\alpha-1)} =: A_k,
\end{equation*}
where $2(\alpha-1)<-1$ and $2(\alpha-1+\delta) < -1$ with our choice of $\delta$, so that $\sum_{k=1}^\infty A_k < \infty$. Thus, we have shown \eqref{eq:Aconv}, which in turn implies that
\begin{equation}\label{eq:c}
D'_n \sim J(\alpha,\kappa,\{b_k\}_{k=\kappa+1}^{\infty}) n^{-(2\alpha+1)}L_g(1/n)^2, \quad n \rightarrow \infty.
\end{equation}

We will now use the obtained asymptotic relations, \eqref{eq:cprimeprimea}, \eqref{eq:cprime}, \eqref{eq:exact-negligible} and \eqref{eq:c}, to complete the proof.
To this end, it will be convenient to introduce a relation $x_n \gg y_n$ for any sequences $\{x_n\}_{n=1}^\infty$ and $\{y_n\}_{n=1}^\infty$ of positive real numbers that satisfy $\lim_{n \rightarrow \infty}\frac{x_n}{y_n} = \infty$. By \eqref{eq:exact-negligible}, we have $D'_n \gg D_n$. Since $2\alpha+1 < 2$, we find that also $D'_n \gg D''_n$, in view of \eqref{eq:cprime}. The assumption $\gamma > -\frac{2\alpha+1}{2\beta + 1}$ is equivalent to $-(2\alpha+1) > \gamma (2\beta +1)$, so by the estimate \eqref{eq:potterzero} for slowly varying functions, we have $D'_n \gg D'''_n$.
 It then follows that $E_n \sim \E[\sigma(0)^2] D'_n$ as $n \rightarrow \infty$. Further, the condition \eqref{eq:sigmab1} implies that $E_n \gg E'_n$. In view of \eqref{eq:minkowski}, we finally find that $\mathbb{E}[| X_n(t) - X(t) |^2] \sim \E[\sigma(0)^2] D'_n$ as $n \rightarrow \infty$, which completes the proof.
\end{proof}

\subsection{Proof of Equation (\ref{eq:Wiener-cov})}\label{ssec:hyper}

In order to prove \eqref{eq:Wiener-cov}, we rely on the following integral identity for the Gauss hypergeometric function ${}_2 F_1$.

\begin{lemma}\label{lem:hyper} For all $\alpha  \in (-1,\infty)$ and $0 \leq a < b$,
\begin{equation*}
\int_0^a (a-x)^\alpha(b-x)^\alpha dx = \frac{a^{\alpha+1}b^\alpha}{\alpha+1} {}_2 F_1 \Big(-\alpha, 1,\alpha+2,\frac{a}{b}\Big).
\end{equation*}
\end{lemma}
\begin{proof}
In the case $a=0$ the asserted identity becomes trivial, so we may assume that $a>0$.
Substituting $y = \frac{x}{a}$, we get
\begin{equation*}
\int_0^a (a-x)^\alpha(b-x)^\alpha dx = a^{\alpha + 1}b^\alpha\int_0^1 (1-y)^\alpha \Big(1- \frac{a}{b}y\Big)^\alpha dy.
\end{equation*}
Using Euler's formula \cite[p.\ 59]{EMOT1953}, we find that
\begin{equation*}
\int_0^1 (1-y)^\alpha \Big(1- \frac{a}{b}y\Big)^\alpha dy = \frac{\Gamma(1)\Gamma(\alpha+1)}{\Gamma(\alpha+2)} {}_2 F_1 \Big(-\alpha, 1,\alpha+2,\frac{a}{b}\Big),
\end{equation*}
where $\Gamma$ denotes the gamma function, observing that this step is valid since $\alpha + 2 > 1 > 0$ and $0 < \frac{a}{b} < 1$ under our assumptions. By the connection between the gamma function $\Gamma$ and the beta function $\mathrm{B}$ (see, e.g., \cite[p.\ 9]{EMOT1953}), we obtain further
\begin{equation*}
\frac{\Gamma(1)\Gamma(\alpha+1)}{\Gamma(\alpha+2)} = \mathrm{B}(1,\alpha+1)= \int_{0}^1 (1-x)^{\alpha} dx= \frac{1}{\alpha+1},
\end{equation*}
concluding the proof.
\end{proof}

\begin{proof}[Proof of Equation \eqref{eq:Wiener-cov}]
Let $\alpha \in (-\frac{1}{2},\frac{1}{2}) \setminus \{0 \}$ and let $j$,\ $k = 2,\ldots,\kappa+1$ be such that $j < k$. By the It\=o isometry, we have
\begin{equation*}
\begin{split}
\Sigma_{j,k} = \E[W^n_{0,j-1}W^n_{0,k-1}] & = \int_{0}^{\frac{1}{n}} \bigg( \frac{j-1}{n} -s\bigg)^\alpha\bigg( \frac{k-1}{n} -s\bigg)^\alpha ds \\ & = \frac{1}{n^{2\alpha+1}} \int_0^1 (j-1-x)^\alpha (k-1-x)^\alpha dx,
\end{split}
\end{equation*}
where we have substituted $x = ns$. The second integral above can now be expressed as
\begin{multline*}
\int_0^1 (j-1-x)^\alpha (k-1-x)^\alpha dx  \\  
\begin{aligned}
 & = \int_0^{j-1} (j-1-x)^\alpha (k-1-x)^\alpha dx
- \int_{0}^{j-2} (j-2-x)^\alpha (k-2-x)^\alpha dx \\
& = \frac{1}{\alpha+1} \Bigg((j-1)^{\alpha+1}(k-1)^\alpha {}_2 F_1 \bigg(-\alpha, 1,\alpha+2,\frac{j-1}{k-1}\bigg) \\
& \quad - (j-2)^{\alpha+1}(k-2)^\alpha {}_2 F_1 \bigg(-\alpha, 1,\alpha+2,\frac{j-2}{k-2}\bigg)\Bigg),
\end{aligned}
\end{multline*}
where the second equality follows from Lemma \ref{lem:hyper}, which is applicable to both integrals on the second line as $0 < j-1< k-1$ and $0 \leq j-2 < k-2$.
\end{proof}

\section*{Acknowledgements}
We would like to thank Heidar Eyjolfsson and Emil Hedevang for useful discussions regarding simulation of $\mathcal{BSS}$ processes and Ulises M\'arquez for assistance with symbolic computation. Our research has been supported by CREATES (DNRF78), funded by the Danish National Research Foundation, 
by Aarhus University Research Foundation (project ``Stochastic and Econometric Analysis of Commodity Markets") and
by the Academy of Finland (project 258042).

{\small 
\bibliographystyle{mikko}
\bibliography{mybib-v9}
}

\end{document}